\documentclass[a4paper,11pt]{amsart}
\usepackage[left=2.7cm,right=2.7cm,top=3.5cm,bottom=3cm]{geometry}

\usepackage{amsthm,amssymb,amsmath,amsfonts,mathrsfs,amscd,dsfont,bm}
\usepackage{stmaryrd}
\usepackage[latin1]{inputenc}
\usepackage[all,cmtip]{xy}
\usepackage{latexsym}
\usepackage{xcolor}
\usepackage{enumerate}
\usepackage{longtable}
\usepackage{mathtools}
\usepackage{comment}

\numberwithin{equation}{section}

\usepackage[pagebackref]{hyperref}

\mathtoolsset{showonlyrefs}

\newfont{\cyr}{wncyr10 scaled 1100}
\newfont{\cyrr}{wncyr9 scaled 1000}

\theoremstyle{plain}
\newtheorem{theorem}{Theorem}[section]

\newtheorem{proposition}[theorem]{Proposition}
\newtheorem{lemma}[theorem]{Lemma}

\theoremstyle{definition}
\newtheorem{definition}[theorem]{Definition}
\newtheorem{assumption}[theorem]{Assumption}

\theoremstyle{remark}
\newtheorem{remark}[theorem]{Remark}

\newcommand{\Q}{\mathds Q}

\newcommand{\Z}{\mathds Z}
\newcommand{\R}{\mathds R}
\newcommand{\C}{\mathds C}

\DeclareMathOperator{\Aut}{Aut}
\DeclareMathOperator{\Frob}{Frob}

\DeclareMathOperator{\Hom}{Hom}

\DeclareMathOperator{\Gal}{Gal}
\DeclareMathOperator{\GL}{GL}

\DeclareMathOperator{\SL}{SL}

\DeclareMathOperator{\im}{im}

\newcommand{\res}{\mathrm{res}}
\newcommand{\cores}{\mathrm{cores}}

\newcommand{\Sha}{\mbox{\cyr{X}}}

\definecolor{Green}{rgb}{0.1,0.9,0.2}

\newcommand{\longmono}{\mbox{\;$\lhook\joinrel\longrightarrow$\;}}

\newfont{\gotip}{eufb10 at 12pt}

\newcommand{\cO}{{\mathcal O}}

\newcommand{\p}{\mathfrak{p}}

\include{thebibliography}

\begin{document}

\title[On the IMC for generalized Heegner classes in a quaternionic setting]{On the Iwasawa main conjecture for generalized Heegner classes in a quaternionic setting}
\author{Maria Rosaria Pati}

\thanks{The author gratefully acknowledges financial support by Projet KUPSUP RIN Emergent 2022 Region Normandie.}

\begin{abstract}
We prove one divisibility relation of the anticyclotomic Iwasawa Main Conjecture for a higher weight ordinary modular form $f$ and an imaginary quadratic field satisfying a  ``relaxed'' Heegner hypothesis. Let $\Lambda$ be the anticyclotomic Iwasawa algebra. Following the approach of Howard and Longo--Vigni, we construct the $\Lambda$-adic Kolyvagin system of generalized Heegner classes coming from Heegner points on a suitable Shimura curve. As its application, we also prove one divisibility relation in the Iwasawa-Greenberg main conjecture for the $p$-adic $L$-function defined by Magrone.
\end{abstract}

\address{Laboratoire de Math\'ematiques Nicolas Oresme, CNRS UMR 6139,
Universit\'e de Caen - Normandie   BP 5186,
14032 Caen Cedex,
France}
\email{maria-rosaria.pati@unicaen.fr}



\maketitle

\tableofcontents

\section{Introduction}
The study of the Iwasawa theory of Galois representations attached to modular forms goes back to Perrin-Riou, who in 1987 formulated an anticyclotomic Iwasawa Main Conjecture for elliptic curves. More precisely, given an \emph{elliptic curve} $E$ over $\Q$ of conductor $N$, a \emph{prime} $p$ of good ordinary reduction for $E$ and an \emph{imaginary quadratic field} $K$ satisfying the Heegner hypothesis relative to $N$ (i.e. all the primes dividing $N$ are split in $K$), she formulated a conjecture on the structure of the $p^\infty$-Selmer group of $E$ over $K_\infty$ as a module over the Iwasawa algebra $\Lambda=\Z_p\llbracket \Gal(K_\infty/K)\rrbracket$. Perrin-Riou's conjecture was proved under mild hypothesis in 2020 by Burungale-Castella-Kim \cite{BCK}.

As one can see from the emphasized words, there are various directions of generalizations of Perrin-Riou's conjecture. For example, one can consider settings in which $p$ is not ordinary and/or the Galois representation attached to $E$ is replaced by the Galois representation attached to a higher weight modular form and/or $K$ does not satisfy the Heegner hypothesis. 

In the present work, we consider a setting in which the imaginary quadratic field $K$ does not satisfy the Hegneer hypothesis, but its ``relaxed'' version. More precisely, let $f$ be a normalized newform of even weight $k\geq 4$ and level $\Gamma_0(N)$, for $N\geq 3$ an integer. Fix a prime $p$ such that $p\nmid N$ and an imaginary quadratic field $K$ such that $(\mathrm{disc}(K),Np)=1$, $p$ is split in $K$, and $K$ satisfies a relaxed Heegner hypothesis relative to $N$. This means that if $N=N^+N^-$ is the factorization of $N$ in which the primes dividing $N^+$ (resp. $N^-$) are split (resp. inert) in $K$, then $N^-$ is a square-free product of an even number of primes. The relaxed setting we consider leads us to work with Heegner objects over a suitable Shimura curve (instead of the modular curve $X_0(N)$).

Fix an embedding $\iota_p\colon \bar \Q\hookrightarrow \bar\Q_p$, and let $\p$ be the prime above $p$ induced by $\iota_p$ in the ring of integers $\cO$ of the number field $\Q_f$ generated by the Fourier coefficients of $f$. We define $\Lambda=\cO_\p\llbracket \Gamma\rrbracket$ the Iwasawa algebra of $\Gamma=\Gal(K_\infty/K)$ over the localization $\cO_\p$ of $\cO$ at $\p$. Let $V=V_{f,\p}(k/2)$ be the self dual twist of the $\p$-adic Galois representation attached to $f$ by Deligne, $T\subseteq V$ the $G_\Q$-stable lattice inside $V$ as defined in \cite{Nek}, and $A=V/T$. The $p$-Selmer group of $E$ over $K_n$ in the higher weight setting is replaced by the Bloch-Kato Selmer group $H^1_\mathrm{BK}(K_n,A)$. The anticyclotomic Iwasawa theory in our setting is the study of the structure as a $\Lambda$-module of the finitely generated compact $\Lambda$-module $\mathcal{X}_\infty:=(\varinjlim_n H^1_\mathrm{BK}(K_n,A))^\vee$, where the superscript $\vee$ denotes as usual Pontryagin dual. Using Bloch-Kato Selmer groups we can define another module over $\Lambda$, the pro-$p$ Bloch-Kato Selmer group of $f$ over $K_\infty$ as $\hat H^1_\mathrm{BK}(K_\infty,T):=\varprojlim_n H^1(K_n,T)$ with respect to the corestriction. It is a finitely generated $\Lambda$-module which contains the element $\kappa_\infty$ coming from the compatibility along the anticyclotomic extensions of $K$ of generalized Heegner classes over a suitable Shimura curve. 
This element $\kappa_\infty$ is expected to be non-torsion, but a proof of this claim is not yet available (it would follow from \cite[Theorem 4.6]{Magrone} whose proof is only announced). In addition to the assumptions we made above on $f$, $p$ and $K$, we also assume that:
\begin{enumerate}
\item $p\nmid 2N h_K d_f$ where $h_K$ denotes the class number of $K$ and $d_f$ is the discriminant of $\Q_f$;\\
\item $f$ is ordinary at $p$, i.e. the $p$-th Fourier coefficient $a_p$ of $f$ is a $p$-adic unit;\\
\item the $p$-adic representation $\rho_{f,\p}\colon G_\Q\rightarrow \Aut(T)\simeq \GL_2(\cO_\p)$ attached to $f$ has ``big image", i.e.
\[ \mathrm{im}(\rho_{f,\p})\supseteq \{ g\in \GL_2(\Z_p)\ :\ \det(g)\in(\Z_p^\times)^{k-1}\}. \]
\end{enumerate}
In this setting we prove the following:

\begin{theorem}\label{main}
If $\kappa_\infty$ is not torsion, then there exists a finitely generated torsion $\Lambda$-module $M_\infty$ such that
\begin{enumerate}
\item $\mathcal{X}_\infty\sim \Lambda \oplus M_\infty \oplus M_\infty$,\\
\item $\mathrm{char}(M_\infty)=\mathrm{char}(M_\infty)^\iota$ and $\mathrm{char}(M_\infty)\supseteq \mathrm{char}(\hat H^1_\mathrm{BK}(K_\infty,T)/\Lambda \kappa_\infty)$
\end{enumerate}
where $\sim$ means pseudo-isomorphism and $\iota\colon \Lambda\rightarrow \Lambda$ is the involution induced by the inversion on $\Gal(K_\infty/K)$.
\end{theorem}

The strategy of the proof is the same as that followed by Longo and Vigni in \cite{LV-Kyoto}, where they proved the same result under the standard Heegner hypothesis, i.e. $N^-=1$, and hence with the element $\kappa_\infty$ constructed using generalized Heegner cycles over the modular curve $X_0(N)$. It is based on the Kolyvagin system machinery of Mazur and Rubin (see \cite{Mazur-Rubin}) as adapted by Howard in \cite{Howard} to modules over an Iwasawa algebra to prove the analogous of our result in the elliptic curve setting. 
The main ingredient of our proof is then the construction of a $\Lambda$-adic Kolyvagin system of generalized Heegner classes given by Abel-Jacobi images of cycles supported on fibers over a suitable Shimura curve. 

Along the years many results on Selmer groups and Shafarevich-Tate groups appeared in literature. If one is interested on the study of these arithmetic objects over the fixed imaginary quadratic field $K$, we could mention the works of Nekov\'{a}\v{r} and Besser, \cite{Nek} and \cite{Besser}, in the classical Heegner setting, and of Elias-de Vera-Piquero \cite{Elias-deVeraPiq} in the quaternionic setting. The present paper deals with the study of Selmer groups along the anticyclotomic tower. It is more in line with the cited works of Howard, Longo--Vigni, and Burungale-Castella-Kim.
In view of the relation between the higher weight specializations of big Heegner points and generalized Heegner cycles in the quaternionic setting (see \cite[Theorem 11.1]{LMW}), our result has also been obtained in \cite[Theorem B (ii)]{Fouquet} under some slightly different hypothesis. Note that, as for Theorem \ref{main}, the result in \textit{loc. cit.} is conditional to the assumption that the higher weight specialization of a certain $\Lambda$-adic big Heegner point is non-torsion.

As an application of Theorem \ref{main}, we prove one inclusion in the Iwasawa-Greenberg main conjecture for the $p$-adic $L$-function defined in \cite{Magrone}. It is well known in fact that there exists an equivalent formulation of the anticyclotomic Iwasawa main conjecture in terms of $p$-adic $L$-functions. In \cite{Magrone}, Magrone generalizes to the quaternionic setting the results of \cite{CH}. For a fixed anticyclotomic Hecke character $\psi$ of infinity type $(k/2,-k/2)$, she defines a $p$-adic $L$-function $\mathscr{L}_{\p,\psi}(f)$ in the Iwasawa algebra $\Lambda^\mathrm{ur}$ of $\Gal(K_\infty/K)$ over $\hat\cO^\mathrm{ur}$, the ring of integers of the completion of the maximal unramified extension of $\Q_{f,\p}$, satisfying an explicit reciprocity law which relates $\mathscr{L}_{\p,\psi}(f)$ to the image under a Perrin-Riou big logarithmic map of the bottom class of our $\Lambda$-adic Kolyvagin system.

Denote by $\hat\psi$ the $p$-adic avatar of $\psi$, and assume that $\hat\psi$ factors through $\Gal(K_\infty/K)$. Let $H^1_\mathrm{BK}(K_\infty,A(\hat\psi^{-1}))=\varinjlim_n H^1_\mathrm{BK}(K_n,A(\hat\psi^{-1}))$ be the Bloch-Kato Selmer group attached to the twist of $A$ by $\hat\psi^{-1}$. With the same hypothesis as in Theorem \ref{main} we prove the following

\begin{theorem}
The $\Lambda$-module $X_\psi:=(H^1_\mathrm{BK}(K_\infty,A(\hat\psi^{-1})))^\vee$ is torsion, and the following inclusion of ideals of $\Lambda^\mathrm{ur}$ holds:
\[ (\mathscr{L}_{\p,\psi}(f))^2\subseteq \mathrm{char}(X_\psi)\otimes \Lambda^\mathrm{ur} .\]
\end{theorem}
As for Theorem \ref{main}, we should mention that there is a great variety of results in different settings of the Iwasawa-Greenberg main conjecture for $p$-adic $L$-functions attached to modular forms. For example the work of Kobayashi-Ota \cite{Kob-Ota} considers the $p$-adic $L$-function of Bertolini-Darmon-Prasanna in the case in which the modular form $f$ is non-ordinary at the prime $p$.

\textbf{Notation.}
We fix an embedding $\bar \Q\hookrightarrow \C$ and denote by $\tau\in G_\Q=\Gal(\bar\Q/\Q)$ the complex conjugation it induces.\\

\textit{Acknowledgements.} We would like to thank Stefano Vigni for helpful discussions on one of his joint works with Matteo Longo, and the anonymous referee for valuable comments and suggestions on an earlier version of the article.\\

\section{Galois representations}\label{repr}

In this section we will recollect some known facts about Galois representations attached to modular forms with the aim of presenting some objects that we will use later. See \cite{Brooks} or \cite{Magrone} for more details about the contents of this section.

\subsection{Galois representation of a modular form}\label{galois}
Let $f$ be a normalized newform of weight $k$ for $\Gamma_0(N)$ and write $\cO$ for the ring of integers of the number field $\Q_f$ generated by the Fourier coefficients of $f$. Denote by $Y_0(N)$ the modular curve parametrizing elliptic curves with a level-$N$ structure and let $j\colon Y_0(N)\hookrightarrow X_0(N)$ be its proper smooth compactification. Write $\pi\colon \mathscr{E}\rightarrow Y(N)$ for the universal elliptic curve with full level $N$ structure, and let $q\colon Y(N)\rightarrow Y_0(N)$ be the projection. Note that $Y_0(N)=\mathcal{B}\backslash Y(N)$, where $\mathcal{B}$ is the Borel subgroup of $\GL_2(\Z/N\Z)/\{\pm 1\}$. Let $\epsilon_\mathcal{B}=(1/\# \mathcal{B})\sum_{g \in \mathcal{B}}g$ be the idempotent corresponding to the trivial representation of $\mathcal{B}$.

Fix a prime $p$ such that $p\nmid 2N$ and let $\p\subseteq \cO$ be the prime above $p$ determined by $\iota_p$. The $\p$-adic representation of $G_\Q$ attached to $f$ by Deligne can be defined as the maximal subspace of
$$H^1_{\text{\'et}}(X_0(N)\otimes \bar{\Q},j_*(\epsilon_\mathcal{B} q_*\mathrm{Sym}^{k-2}(\mathbf{R}^1\pi_*\Q_p)))\otimes_{\Q_p}\Q_{f,\p}$$
on which the Hecke operators $T_\ell$ for all $\ell\nmid Np$ act as multiplication by the Fourier coefficient $a_\ell(f)$. It is a $2$-dimensional $\Q_{f,\p}$-vector space, that we will denote by $V_{f,\p}$, equipped with a continuous action of $G_\Q$, unramified outside the primes dividing $Np$ and such that the characteristic polynomial of a geometric Frobenius at $\ell\nmid Np$ is the Hecke polynomial
$$X^2-a_\ell(f) X+\ell^{k-1}.$$
As proved by Scholl (\cite{Scholl}) for modular curves parametrizing elliptic curves with a full level structure, and extended by Nekov\'a\v{r} to $\Gamma_0$-types level structures (\cite[\S 3]{Nek} and \cite[\S 2.5]{Nek2}), the Galois representation $V_{f,\p}$ can be obtained as the $\p$-adic realization of a suitable motive attached to $f$. Roughly speaking, this means that $V_{f,\p}$ lives in the $p$-adic \'etale cohomology of a suitable variety. More precisely, let $\bar{\pi}\colon \bar{\mathscr{E}}\rightarrow X(N)$ be the universal generalized elliptic curve with full level $N$ structure over the modular curve $X(N)$ and let $\bar{\mathscr{E}}^{k-2}$ be the $(k-2)$-fold fibre product of $\bar{\mathscr{E}}$ with itself over $X(N)$. If $k\geq 4$ then $\bar{\mathscr{E}}^{k-2}$ is singular. Denote by $\mathcal{W}_{k-2}$ the canonical desingularization of $\bar{\mathscr{E}}^{k-2}$ constructed by Deligne; it is the so-called $(k-1)$-dimensional Kuga-Sato variety over $X(N)$. An extension due to Nekov\'a\v{r} of the main result of Scholl (\cite{Scholl}) says that there is a canonical isomorphism
$$H^1_{\text{\'et}}(X_0(N)\otimes \bar{\Q},j_*(\epsilon_\mathcal{B} q_*\mathrm{Sym}^{k-2}(\mathbf{R}^1\pi_*\Q_p)))\overset{\sim}{\longrightarrow}\epsilon_\mathcal{B}\epsilon_S H^{k-1}_{\text{\'et}}(\mathcal{W}_{k-2}\otimes \bar{\Q},\Q_p),$$
where $\epsilon_S$ is a suitable projector acting on the vector space $H^{k-1}_{\text{\'et}}(\mathcal{W}_{k-2}\otimes \bar{\Q},\Q_p)$.
Then, putting $\epsilon=\epsilon_\mathcal{B}\epsilon_S$ the representation $V_{f,\p}$ is a \emph{subspace} of $\epsilon H^{k-1}_{\text{\'et}}(\mathcal{W}_{k-2}\otimes \bar{\Q},\Q_p)\otimes \Q_{f,\p}$. Poincar\'e duality on $\mathcal{W}_{k-2}$ gives rise to a $G_\Q$-equivariant, non-degenerate, alternating pairing
\begin{equation}\label{poinc on V}
e\colon V_{f,\p}\times V_{f,\p}\longrightarrow \Q_{f,\p}(1-k)
\end{equation}
for which the twist $V:=V_{f,\p}(k/2)$ is self dual, in the sense that $V\simeq V^*(1)$ where $V^*$ is the contragredient representation of $V$. As explained in \cite[\S 2.1.2]{PPV}, the representation $V$ is isomorphic to the self dual twist of $V_{f,\p}^*$ and hence it can be realized as a \emph{quotient} of $\epsilon H^{k-1}_{\text{\'et}}(\mathcal{W}_{k-2}\otimes \bar{\Q},\Q_p(k/2))$. We then have a natural projection
$$\epsilon H^{k-1}_{\text{\'et}}(\mathcal{W}_{k-2}\otimes \bar{\Q},\Q_p(k/2))\otimes \Q_{f,\p}\longrightarrow V$$
which is $G_\Q$-equivariant.

As explained by Nekov\'a\v r, there exists an integral version of $V$. Let $\cO_\p$ be the valuation ring of $\Q_{f,\p}$. By  \cite[Proposition 3.1]{Nek} under the assumption $p\nmid (k-2)!\phi(N)$, there exists a $G_\Q$-stable lattice $T$ in $V$ and a $G_\Q$-equivariant map
$$r_f\colon \epsilon H^{k-1}_{\text{\'et}}(\mathcal{W}_{k-2}\otimes \bar{\Q},\Z_p(k/2))\otimes \cO_\p\longrightarrow T.$$
In a subsequent paper, \cite[\S 6.5]{Nek2}, the author explains why it is possible to remove the assumption $p\nmid (k-2)!\phi(N)$.
Since $e$ is non-degenerate, the lattice dual of $T$, i.e. $\{x\in V\ |\ e(x,T)\subseteq\cO_\p(1)\}$, is isomorphic to $T$ (see \cite[Proposition 3.1]{Nek}), hence $e$ induces a pairing that we continue to denote by $e$
\begin{equation}\label{weil pairing}
e\colon T\times T\longrightarrow \cO_\p(1)=\cO_\p\otimes \Z_p(1)
\end{equation}
which is perfect, alternating and Galois equivariant (note that $\tau\in G_\Q$ acts on the target as the inverse since $\Z_p(1)$ are roots of unity).

\subsection{Galois representation of a modular form in a quaternionic setting}\label{Shimura}
In what follows we will consider a factorization of $N$ as $N=N^+ N^-$ such that $N^-$ is a square-free product of an even number of primes.

Let $B$ be the (indefinite) quaternion algebra over $\Q$ with discriminant $N^-$. Fix a maximal order $\cO_B$ in $B$ and Eichler order $\mathcal{R}\subseteq \cO_B$ of level $N^+$ (they are unique up to conjugation by elements of $B^\times$). Associated to these data is a Shimura curve $X_{N^+}/\Q$, which is the coarse moduli scheme classifying abelian surfaces with quaternionic multiplication by $\cO_B$ and a level $N^+$-structure (see \cite[\S 2.1]{Elias-deVeraPiq} for more details).

Since $B$ is split at infinity, we can fix an isomorphism $\iota_\infty \colon B\otimes \R\rightarrow M_2(\R)$. Over the complex numbers, $X_{N^+}^{an}(\C)$ is identified with the compact Riemann surface $\Gamma_{0,N^+}\backslash \mathcal{H}$ as complex algebraic curves, where $\Gamma_{0,N^+}\subseteq \SL_2(\R)$ is the image under $\iota_\infty$ of the group $\mathcal{R}_1^\times$ of units of reduced norm $1$ in the Eichler order $\mathcal{R}$.

Let $M_k(X_{N^+},\C)$ be the space of modular forms on $X_{N^+}$ (see \cite[\S 3.1]{Elias-deVeraPiq}). The Shimura curve $X_{N^+}$ comes equipped with a ring of Hecke correspondences, which induces an action of Hecke operators on $M_k(X_{N^+},\C)$.


 The Jacquet-Langlands correspondence is a canonical (up to scaling) isomorphism
$$S_k(\Gamma_0(N),\C)^{N^-\mathrm{-new}}\overset{\simeq}{\longrightarrow} M_k(X_{N^+},\C)$$
which is compatible with the action of the Hecke operators and Atkin-Lehner involutions.

Hence, in particular, to each eigenform $f\in S_k(\Gamma_0(N))^{N^-\mathrm{-new}}$ there corresponds a unique quaternionic form in $M_k(X_{N^+})$ having the same Hecke eigenvalues as $f$ for the good Hecke operators $T_\ell$ ($\ell\nmid N$) and the Atkin-Lehner involutions.

Let $f_B\in M_k(X_{N^+},\C)$ be the modular form on $X_{N^+}$ corresponding to the normalized newform $f$ of weight $k$ for $\Gamma_0(N)$ fixed in the previous section. A construction similar to that of $V_{f,\p}$ for $f$ can be done for $f_B$ using the universal abelian surface $\mathcal{A}$ over an auxiliary Shimura curve which is a fine moduli scheme projecting onto $X_{N^+}$.


We will denote by $X_{1,N^+}$ the Shimura curve of level $V_1(N^+)$ (see for example \cite[\S 3]{Kassaei}). It is the fine moduli scheme over $\Z[1/N]$ which represents the moduli problem:
$$S\ \mapsto \{\text{isomorphism classes of abelian surfaces with QM and level}\ V_1(N^+)\ \text{structure}\}.$$ 

Since $B$ is split at the primes dividing $N^+$, and $\cO_B$ is a maximal order, there is an identification $\eta\colon \cO_B\otimes \Z/N^+\Z\rightarrow M_2(\Z/N^+\Z)$ such that the Eichler order $\mathcal{R}$ consists of those elements $x$ in $\cO_B$ for which $\eta(x)$ is a upper triangular matrix. The group $\mathcal{R}_1^\times$ admits a canonical map to $(\Z/N^+\Z)^\times$ sending $x$ to $d$, where $d$ is the bottom right entry in $\eta(x)=\begin{pmatrix} a & b\\ 0 & d\end{pmatrix}$. Let $\Gamma_{1,N^+}$ be the image under $\iota_\infty$ of the kernel of this map. 
The complex points of the complex algebraic curve $X_{1,N^+}^{an}$ are naturally identified with the compact Riemann surface $\Gamma_{1,N^+}\backslash\mathcal{H}$.

Let $\pi\colon\mathcal{A}\rightarrow X_{1,N^+}$ be the universal abelian surface over $X_{1,N^+}$,
 and let $\mathcal{A}^{\frac{k-2}{2}}$ be the $(k-2)/2$-fold fibre product of $\mathcal{A}$ with itself over $X_{1,N^+}$. In analogy to the elliptic case, we refer to it as the $(k-1)$-dimensional Kuga-Sato variety over the Shimura curve $X_{1,N^+}$. Note that, since $X_{1,N^+}$ has no cusps, the product variety $\mathcal{A}^{\frac{k-2}{2}}$ does not require any desingularization.
The group $\mathcal{B}=\Gamma_{0,N^+}/\Gamma_{1,N^+}\simeq (\Z/N^+\Z)^\times$ acts canonically as $X_{N^+}$-automorphisms on $\mathcal{A}^{\frac{k-2}{2}}$, and hence we can define the projector $\epsilon_\mathcal{B}$ exactly as in the previous section. The $\p$-adic representation $V_{f_B,\p}$ of $G_\Q$ attached to $f_B$ is the two-dimensional vector space given by the maximal subspace of
$$\epsilon_\mathcal{B}\epsilon_B H_{\text{\'et}}^{k-1}(\mathcal{A}^{\frac{k-2}{2}}\otimes \bar\Q,\Q_p)\otimes_{\Q_p}\Q_{f,\p}$$
on which the Hecke operators $T_\ell$ for all $\ell\nmid Np$ act as multiplication by $a_\ell(f)$. Here $\epsilon_B$ is a suitable projector defined for example in \cite[Theorem 6.1]{Brooks}. 
More generally, $\epsilon_\mathcal{B}\epsilon_B H_{\text{\'et}}^{k-1}(\mathcal{A}^{\frac{k-2}{2}}\otimes \bar\Q,\Q_p)$ is a free Hecke-module of rank $2$ (the arguments are the same as in \cite[\S 5]{IS}).

The Eichler-Shimura relations over Shimura curves imply that the traces of the Frobenii at primes $\ell\nmid Np$ acting on $V_{f_B,\p}$ is given by the action of $T_\ell$ on $f_B$, and hence by multiplication by $a_\ell(f)$. This means that the Frobenii at all primes $\ell\nmid Np$ act in the same way on $V_{f_B,\p}$ and on $V_{f,\p}$. But then, a Chebotarev's density argument implies that
$$V_{f,\p}\simeq V_{f_B,\p}$$
as Galois representations and Hecke-modules. Therefore, putting $\epsilon=\epsilon_\mathcal{B}\epsilon_B$ and tensoring with $\Q_\p(k/2)$ we find the projection map
$$\epsilon H_{\text{\'et}}^{k-1}(\mathcal{A}^{\frac{k-2}{2}}\otimes \bar\Q,\Q_p(k/2))\otimes \Q_{f,\p}\longrightarrow V.$$
For reasons that will appear clear later, we are interested in an integral version of this map. 
Since $p\nmid 2N$, the same construction of $V_{f_B,\p}$ can be done with the constant sheaf $\Q_p$ replaced by $\Z_p$. As in the previous section, at a first stage one should assume that $p\nmid \phi(N^+)$ but in fact, as above, by suitably choosing the distinguished $\cO_\p$-lattice $T_{f_B,\p}$ in $V_{f_B,\p}$ one can drop this assumption. Clearly, one has $T_{f_B,\p}(k/2)\simeq T$ and a $G_\Q$-equivariant map
$$r_{B,f}\colon \epsilon H_{\text{\'et}}^{k-1}(\mathcal{A}^{\frac{k-2}{2}}\otimes \bar\Q,\Z_p(k/2))\otimes \cO_\p\longrightarrow T.$$
In this sense we view the Galois representation $T$ attached to $f$ in a quaternionic setting.

\subsection{Generalized Kuga-Sato varieties fibered over a Shimura curve}\label{imaginary field}
In this section we will define a kind of \emph{generalized} Kuga-Sato variety over the Shimura curve $X_{1,N^+}$. It is the variety on which the algebraic cycles that we will use to construct our Kolyvagin system lie. 

Fix once and for all an imaginary quadratic field $K$ of discriminant coprime to $Np$ in which $p$ splits, and let $\cO_K$ its ring of integers. For simplicity we assume that $\cO_K^\times=\{\pm 1\}$; a condition which is satisfied as soon as $K\neq \Q(i)$ and $K\neq \Q(\sqrt{-3})$. The field $K$ determines a factorization of $N$ as $N=N^+N^-$ where $N^+$ is the product of the primes of $N$ that are split in $K$ and $N^-$ is the product of the primes of $N$ that are inert in $K$.
Throughout the paper we will do the following
\begin{assumption}\label{assumption}
$N^-$ is a square-free product of an even number of primes.
\end{assumption}
With this factorization we can consider the Shimura curve $X_{1,N^+}$ and its universal abelian surface $\pi\colon \mathcal{A}\rightarrow X_{1,N^+}$ as defined in sec. \ref{Shimura}.

Fix once and for all an abelian surface $A$ with quaternionic multiplication (by $\cO_B$), level $V_1(N^+)$-structure and complex multiplication by $\cO_K$, defined over the Hilbert class field $H$ of $K$. Since $p$ splits in $K$, the surface $A$ is ordinary at $p$, which means that it has good ordinary reduction at $\p$ and the $p$-divisible group of its reduction is isomorphic to the self-product of the $p$-divisible group of an ordinary elliptic curve.

We define our generalized Kuga-Sato variety over $X_{1,N^+}$ as the $(2k-3)$-dimensional variety:
$$Y_k:=\mathcal{A}^{\frac{k-2}{2}}\times A^{\frac{k-2}{2}}.$$

\begin{remark}
When $N^-=1$ the Shimura curve $X_{1,N}=X_1(N)$ is the usual modular curve associated to $\Gamma_1(N)$, $\mathcal{A}$ is the square of the universal generalized elliptic curve with $\Gamma_1(N)$-level structure and $A$ is the square of a fixed elliptic curve with complex multiplication by $\cO_K$. Therefore we get exactly the variety $X_{k-2}$ defined in \cite[\S 2.2]{BDP}.
\end{remark}
We will consider the projector $\epsilon_1\in \text{Corr}_{X_{1,N^+}}(Y_k,Y_k)$ defined in \cite[\S 6.1]{Brooks}, and we will see the projector $\epsilon_\mathcal{B}$ defined in Sec. \ref{Shimura} as an idempotent in the ring $\mathrm{Corr}_{X_{N^+}}(Y_k)$ acting trivially on $A^\frac{k-2}{2}$ (see also \cite[\S 3]{LP} and \cite[Definition 6.4]{Masdeu}). The projector $\epsilon_{Y_k}:=\epsilon_\mathcal{B}\epsilon_1$ has the following effect on the $p$-adic cohomology of $Y_k$ (\cite[Proposition 6.5]{Brooks}):
\begin{equation}\label{cohom}
\epsilon_{Y_k} H^i_{\text{\'et}}(Y_k\otimes \bar\Q,\Z_p)=
\begin{cases}
0 & \text{if}\ i\neq 2k-3\\
\epsilon H^{k-1}_{\text{\'et}}(\mathcal{A}^{\frac{k-2}{2}}\otimes \bar\Q,\Z_p)\otimes \text{Sym}^{k-2}eH^1_{\text{\'et}}(A\otimes\bar\Q,\Z_p) & \text{if}\ i=2k-3
\end{cases}
\end{equation}

\subsection{$p$-adic Abel-Jacobi maps and Heegner classes}\label{class-Paola}
Following \cite[\S 5.5-5.6]{Magrone}, for any field $F$ containing $H$ we consider the $p$-adic Abel-Jacobi map over $Y_k$:
$$\text{AJ}_{p,F}\colon \text{CH}^{k-1}(Y_k/F)_0\longrightarrow H^1(F,H^{2k-3}_{\text{\'et}}(Y_k\otimes\bar\Q,\Z_p(k-1)))$$
where $\text{CH}^{k-1}(Y_k/F)_0$ denotes the Chow group of rational equivalence classes of codimension $k-1$ cycles on $Y_k$ defined over $F$ that are homologically trivial, i.e. belong to the kernel of the cycle class map in $p$-adic cohomology.

Let $c$ be a positive integer prime to $N^+$ and $K[c]$ be the ring class field of $K$ of conductor $c$.
By \eqref{cohom}, the group $\text{CH}^{k-1}(Y_k/K[c])_0$ contains $\epsilon_{Y_k}\text{CH}^{k-1}(Y_k/K[c])$, hence we can compute $\text{AJ}_{p,K[c]}$ on the cycles
$$\Delta_\varphi:=\epsilon_{Y_k}(\text{Graph}(\varphi)^{\frac{k-2}{2}})\in\epsilon_{Y_k}\text{CH}^{k-1}(Y_k/K[c])$$
indexed by isogenies from $A$ to abelian surfaces with $QM$, level $V_1(N^+)$-structure over $K[c]$ and complex multiplication by the order $\cO_c=\Z+c\cO_K$. In fact, if $\varphi\colon A\rightarrow A'$ is an isogeny of abelian surfaces with $QM$ and level $V_1(N^+)$-structure over $K[c]$ we can consider its graph $\text{Graph}(\varphi)\subseteq A\times A'$. Let $x\in X_{1,N^+}(K[c])$ be the $\mathrm{CM}$ point of $X_{1,N^+}$ corresponding to the abelian surface $A'$ with $QM$ and level $V_1(N^+)$-structure. There is an embedding $A'=\mathcal{A}_x\hookrightarrow \mathcal{A}$ and hence
$$\text{Graph}(\varphi)^{\frac{k-2}{2}}\subseteq (A\times A')^{\frac{k-2}{2}}\simeq (A')^{\frac{k-2}{2}}\times A^{\frac{k-2}{2}}\subseteq Y_k.$$
The cycles $\Delta_\varphi$ are called \emph{generalized Heegner cycles} of conductor $c$ over the Shimura curve $X_{N^+}$.

As said above, we can compute $\text{AJ}_{p,K[c]}$ on $\Delta_\varphi$ obtaining a class in $H^1(K[c],\epsilon_{Y_k}H^{2k-3}_{\text{\'et}}(Y_k\otimes\bar\Q,\Z_p(k-1)))$. Finally, by \eqref{cohom}, composing with the map induced by $r_{B,f}$ we define the \emph{generalized Heegner class} attached to $f$ and $\varphi$ as 
$$z_{f,\varphi}:=(r_{B,f}\otimes \mathrm{id})(\text{AJ}_{p,K[c]})(\Delta_\varphi)\in H^1\left(K[c],T\otimes\text{Sym}^{k-2}eH^1_{\text{\'et}}\left(A\otimes\bar\Q,\Z_p\left(\frac{k-2}{2}\right)\right)\right).$$
If in particular $\varphi$ is induced by the multiplication-by-$c$ map on the complex uniformization of $A$, we will write $z_{f,c}$ instead of $z_{f,\varphi}$. We will consider integers $c$ of the form $c_0 p^s$, where $c_0$ is any positive integer such that $(c_0,pN^+)=1$. Finally, as explained in \cite[\S 5.9]{Magrone}, one can define the trivial component of $z_{f,c_o p^s}$, 
\begin{equation}\label{class}
z_{f,1,c_0 p^s}\in H^1(K[c_0 p^s],T),
\end{equation}
as a suitable twist of $z_{f,c_0 p^s}$ (see \cite[(5.5)]{Magrone}).

\section{Selmer groups and the Kolyvagin method}
In this section we recollect some known facts about Selmer groups and the Kolyvagin method to bound them. We will follow mainly \cite{LV-Kyoto} which is the generalization to the higher weight setting of the elliptic curve case covered by Howard in \cite{Howard}. We invite the reader to consult these papers for more details about the contents of this section.

\subsection{Selmer structures}
Let $p$ be a prime number and $R$ a complete, noetherian, local ring with finite residue field of characteristic $p$. Let $K$ be an imaginary quadratic field and $T$ a $R[G_K]$-module whose $G_K$-action is unramified outside a finite set $\Sigma$ of primes of $K$.

A Selmer structure $\mathcal{F}$ on $T$ over $K$ is the datum of a finite set $\Sigma(\mathcal{F})$ of places containing $\Sigma$, all the primes above $p$ and all the archimedean primes, and for each $v\in \Sigma(\mathcal{F})$ the choice of a local condition, that is a subgroup $H^1_\mathcal{F}(K_v,T)$ of $H^1(K_v,T)$.\\
At primes $v\notin \Sigma(\mathcal{F})$ we will write $H^1_\mathcal{F}(K_v,T)$ to denote the unramified or finite condition
$$H^1_\mathcal{F}(K_v,T):=H^1_{\mathrm{unr}}(K_v,T):=\mathrm{ker}(H^1(K_v,T)\longrightarrow H^1(K_v^{\mathrm{unr}},T)).$$

Given a Selmer structure $\mathcal{F}$ on $T$ we define the associated Selmer group $H^1_\mathcal{F}(K,T)\subseteq H^1(K,T)$ as
$$H^1_{\mathcal{F}}(K,T):=\mathrm{ker}\left(H^1(K,T)\longrightarrow \prod_v \frac{H^1(K_v,T)}{H^1_\mathcal{F}(K_v,T)}\right),$$
where the map in parenthesis is induced by the product of all localization maps
$$\mathrm{loc}_v\colon H^1(K,T)\longrightarrow H^1(K_v,T)$$
for every prime $v$ of $K$, and natural projection to quotients. In other words, $H^1_\mathcal{F}(K,T)$ is nothing more than the set of classes in $H^1(K,T)$ whose localization lives in $H^1_\mathcal{F}(K_v,T)$ for every prime $v$ of $K$.

\subsection{Bloch-Kato Selmer groups of a modular form}
The Bloch-Kato Selmer group of a modular form is the Selmer group of the Galois representation attached to the modular form associated to the Selmer structure given by the Bloch-Kato local conditions described below. 

More precisely, let $f$ be a normalized eigenform as in section \ref{galois} and let $V$ be the $2$-dimensional $\Q_{f,\p}$-vector space with a continuous action of $G_K$ defined there. We also recall its $G_K$-stable $\cO_\p$-lattice $T$ and define the quotient $A:=V/T$. 
From now on, we will work under the following
\begin{assumption}
$f$ is ordinary at $p$, i.e. the $p$-th Fourier coefficient $a_p$ is a $p$-adic unit.
\end{assumption}

Let $L$ be any extension of $K$. We will give a Selmer structure $\mathcal{F}$ on $V$ over $L$ and, abusing notation, we will denote by the same symbol $\mathcal{F}$ the Selmer structure on $T$ and $A$ induced in the following way. For a prime $v$ of $L$, we let $H^1_\mathcal{F}(L_v,T)$ (resp. $H^1_\mathcal{F}(L_v,A)$) be the inverse image (resp. the image) of $H^1_\mathcal{F}(L_v,V)$ under the natural map $H^1(L_v,T)\rightarrow H^1(L_v,V)$ (resp. $H^1(L_v,V)\rightarrow H^1(L_v,A)$).

The \emph{Bloch-Kato Selmer group $H^1_{\mathrm{BK}}(L,V)$ of $f$ over $L$} is the Selmer group associated to the following Selmer structure
$$H^1_{\mathrm{BK}}(L_v,V)=
\begin{cases}
\mathrm{ker}(H^1(L_v,V)\rightarrow H^1(L_v,V\otimes_{\Q_p}B_\mathrm{cris})) & \text{if}\ v\mid p,\\
H^1_{\mathrm{unr}}(L_v,V)=\mathrm{ker}(H^1(L_v,V)\rightarrow H^1(L_v^\mathrm{unr},V)) & \text{if}\ v\nmid p
\end{cases}
$$
where $B_{\mathrm{cris}}$ is the Fontaine's ring of crystalline periods.

If $E/L$ is finite extension then restriction and corestriction on cohomology groups induce maps between Bloch-Kato Selmer groups, that is
$$\mathrm{res}_{E/L}\colon H^1_\mathrm{BK}(L,V)\longrightarrow H^1_\mathrm{BK}(E,V),\qquad \mathrm{cores}_{E/L}\colon H^1_\mathrm{BK}(E,V)\longrightarrow H^1_\mathrm{BK}(L,V).$$

\subsubsection{The Bloch-Kato Selmer group of $f$ over the anticyclotomic $\Z_p$-extension of $K$}\label{Zp-ext}

Let $K$ be the imaginary quadratic field fixed in section \ref{imaginary field} and assume in addition that $p\nmid h_K=|\Gal(H/K)|$. Recall that for every integer $n\geq 1$ we denote by $K[n]$ the ring class field of $K$ of conductor $n$. Note that $K[1]$ is the Hilbert class field $H$ of $K$. Since $p$ is unramified in $K$, we also have $p\nmid |\Gal(K[p]/H)|$ (use the class number formula for example); this combined with $p\nmid h_K$ implies that $p$ does not divide the order of $\Gal(K[p]/K)$.

It is well known that $\Gamma_m:=\Gal(K[p^{m+1}]/K[p])\simeq \Z/p^m\Z$ for every $m\geq 0$; hence there exists a subgroup $\Delta$ of $\Gal(K[p^{m+1}]/K)$ isomorphic to $\Gal(K[p]/K)$) such that
$$\Gal(K[p^{m+1}]/K)\simeq \Gamma_m\times \Delta.$$
For every $m\geq 0$ define $K_m$ as the subfield of $K[p^{m+1}]$ that is fixed by $\Delta$, so that
$$\Gal(K_m/K)\simeq \Gamma_m\simeq \Z/p^m\Z.$$
The anticyclotomic $\Z_p$-extension of $K$ is the field $K_\infty:=\cup_{m\geq 0} K_m$. Equivalently, it is the unique $\Z_p$-extension of $K$ which is Galois over $\Q$ and such that the involution in $\Gal(K/\Q)$ acts by conjugation on its Galois group over $K$ as the inversion.

Set
$$\Gamma_\infty:=\varprojlim_m \Gamma_m=\Gal(K_\infty/K)\simeq \Z_p,$$
where the inverse limit is taken with respect to the natural projections $\Gamma_{m+1}\rightarrow\Gamma_m$. Moreover, for every $m\geq 1$ set $\Lambda_m:=\cO_\p[\Gamma_m]$ and define the Iwasawa algebra
$$\Lambda:=\varprojlim_m \Lambda_m=\cO_\p\llbracket\Gamma_\infty\rrbracket.$$
Let $\gamma_\infty=(\gamma_1,\dots,\gamma_m,\dots)$ be a topological generator of $\Gamma_\infty$, that is a coherent sequence in which each element $\gamma_m$ is a generator of $\Gamma_m$ for every $m\geq 1$. 

For a $\cO_\p$-module $M$ we write $M^\vee:=\Hom_{\cO_\p}^{\mathrm{cont}}(M,\Q_{f,\p}/\cO_\p)$ for its Pontryagin dual, equipped with the compact-open topology (here $\Hom^{\mathrm{cont}}_{\cO_\p}$ denotes continuous homomorphisms of $\cO_\p$-modules and $\Q_{f,\p}/\cO_\p$ has the discrete topology). 
If $M$ is also a continuous $\Lambda$-module, then the dual $M^\vee$ inherits a structure of continuous $\Lambda$-module. 

Define the \emph{pro-$p$ Bloch-Kato Selmer group $\hat H^1_\mathrm{BK}(K_\infty,T)$ of $f$ over $K_\infty$} as the compact $\Lambda$-module
$$\hat H^1_\mathrm{BK}(K_\infty,T):=\varprojlim_m H^1_\mathrm{BK}(K_m,T),$$
where the inverse limit is taken with respect to the corestriction maps.

On the other hand, we define the \emph{discrete Bloch-Kato Selmer group $H^1_\mathrm{BK}(K_\infty,A)$} as
$$H^1_\mathrm{BK}(K_\infty,A):=\varinjlim_m H^1_\mathrm{BK}(K_m,A),$$
where the injective limit is taken with respect to the restriction maps. Let
$$\mathcal{X}_\infty:=H^1_\mathrm{BK}(K_\infty,A)^\vee:=\Hom^{\mathrm{cont}}_{\cO_\p}(H^1_\mathrm{BK}(K_\infty,A),\Q_{f,\p}/\cO_\p)$$
be its Pontryagin dual. As proved in \cite[Proposition 2.5]{LV-Kyoto} the $\Lambda$-module $\mathcal{X}_\infty$ is finitely generated. Our main result describes its structure.

\subsection{Greenberg's Selmer groups of a modular form}
As the Bloch-Kato Selmer group, the Greenberg's Selmer group of a modular form is the Selmer group of the Galois representation attached to the modular form associated to the Selmer structure given by the Greenberg local conditions that we now describe. Let $L$ be any extension of $K$. We will give a Selmer structure on $V$. 

Since we are assuming that $f$ is ordinary at $p$, for every prime $v$ of $L$ above $p$ the restriction of $V$ to $G_{L_v}$ admits a filtration of $G_{L_v}$-modules
$$0\longrightarrow \mathrm{Fil}_v^+(V)\longrightarrow V \longrightarrow \mathrm{Fil}_v^-(V) \longrightarrow 0,$$
where $\mathrm{Fil}_v^{\pm}(V)$ are both one-dimensional over $\Q_{f,\p}$ and the inertia subgroup acts on $\mathrm{Fil}_v^-(V)$ via $\chi_\mathrm{cyc}^{-(k-2)/2}$. 

The \emph{Greenberg Selmer group $H^1_\mathrm{Gr}(L,V)$ of $f$ over $L$} is the Selmer group associated to the following Selmer structure:
$$H^1_\mathrm{Gr}(L_v,V)=
\begin{cases}
H^1_\mathrm{ord}(L_v,V):=\ker(H^1(L_v,V)\rightarrow H^1(L_v,\mathrm{Fil}_v^-(V))) & \text{if}\ v\mid p\\
H^1_\mathrm{unr}(L_v,V) & \text{if}\ v\nmid p
\end{cases}$$
Note that 
$$H^1_\mathrm{ord}(L_v,V)=\im(H^1(L_v,\mathrm{Fil}_v^+(V))\rightarrow H^1(L_v,V)).$$

Put $\mathrm{Fil}_v^\pm(T):=T\cap\mathrm{Fil}_v^\pm(V)$ and $\mathrm{Fil}_v^\pm(A):=\mathrm{Fil}_v^\pm(V)/\mathrm{Fil}_v^\pm(T)$. Recall the perfect, alternating, Galois equivariant pairing in \eqref{weil pairing}, and define
\begin{align}\label{pairing}
(\ ,\ )\colon T\times T & \longrightarrow \cO_\p(1)\\
(x,y) & \longmapsto e(x,y^\tau).
\end{align}
Then, $(\ ,\ )$ is a perfect, symmetric pairing. 
Moreover, $(\ ,\ )$ satisfies $(x,y)^\sigma=(x^\sigma,y^{\tau\sigma\tau^{-1}})$ for every $x,y\in T$ and $\sigma\in G_K$.
By definition of dual lattice, the pairing \eqref{pairing} gives rise to a pairing
\begin{equation}\label{pairing2}
(\ ,\ )\colon T\times A\longrightarrow (\Q_{f,\p}/\cO_\p)(1)=\mu_{p^\infty}
\end{equation}
still denoted by the same symbol. 

One can define $H^1_\mathrm{Gr}(L,T)$ and $H^1_\mathrm{Gr}(L,A)$ using the analogous local conditions of $V$.

\subsubsection{Greenberg's Selmer group of $\Lambda$-adic representations}\label{Greenberg Lamb-adic}
 We want to define a Greenberg type Selmer structure on the $\Lambda$-adic deformation of $T$, that is on $\Lambda$-module defined by
 $$\mathbf{T}:=T\otimes_{\cO_\p}\Lambda.$$
 More concretely, 
 $$\mathbf{T}=\varprojlim \mathrm{Ind}_{K_m/K}(T),$$ 
 where the inverse limit is taken with respect to corestriction maps. Here we recall that, for any $G_K$-module $M$ and any finite extension $L/K$ the symbol $\mathrm{Ind}_{L/K}(M)$ denotes the $G_K$-module whose elements are functions $f\colon G_K\rightarrow M$ such that $f(\sigma x)=\sigma f(x)$ for all $x\in G_K$ and $\sigma \in G_L$. The action of $G_K$ on $\mathrm{Ind}_{L/K}(M)$ is given by $(f^\sigma)(x)=f(x\sigma)$. Since $M$ is a $G_K$-module (not only a $G_L$-module), $\mathrm{Ind}_{L/K}(M)$ is also endowed with a left action of $\Gal(L/K)$ defined by $(\gamma\cdot f)(x)=\tilde \gamma f(\tilde\gamma^{-1}x)$, where $\tilde\gamma\in G_K$ is any lift of $\gamma\in\Gal(L/K)$. One can easily see that the action of $G_K$ and $\Gal(L/K)$ commutes. There are maps
 $$\cores_m\colon \mathrm{Ind}_{K_{m+1}/K}(M)\rightarrow \mathrm{Ind}_{K_m/K}(M)\quad\text{and}\quad \res_m\colon \mathrm{Ind}_{K_m/K}(M)\rightarrow \mathrm{Ind}_{K_{m+1}/K}(M)$$
 defined by $\cores_m(f)=\sum_{\gamma\in\Gal(K_{m+1}/K_m)}\gamma\cdot f$, and $\res_m(f)=f$ on which $\Gal(K_{m+1}/K)$ acts through the quotient $\Gal(K_m/K)$.
 
 Define 
 $$\mathbf{A}:=\varinjlim \mathrm{Ind}_{K_m/K}(A)$$
 with the limit taken with respect to restriction maps, and let
 \begin{align*} 
\mathrm{Fil}_v^+(\mathbf{T}) &:=\mathrm{Fil}_v^+(T)\otimes_{\cO_\p}\Lambda\simeq \varprojlim \mathrm{Ind}_{K_m/K}(\mathrm{Fil}_v^+(T))\\
\mathrm{Fil}_v^+(\mathbf{A}) &:=\varinjlim \mathrm{Ind}_{K_m/K}(\mathrm{Fil}_v^+(A)).
\end{align*}
As observed in \cite[\S 3.3]{LV-Kyoto}, the pairing \eqref{pairing2} induces a perfect, $G_K$-equivariant pairing
\begin{equation}\label{lambda-pairing}
(\ ,\ )_\infty\colon \mathbf{T}\times \mathbf{A}\longrightarrow \mu_{p^\infty}
\end{equation}
such that $(\lambda x,y)_\infty=(x,\iota(\lambda)y)$ for all $x\in\mathbf{T}$, $y\in\mathbf{A}$, $\lambda\in \Lambda$, where $\iota$ is the $\cO_\p$-linear involution of $\Lambda$ induced by $\gamma_\infty\mapsto \gamma_\infty^{-1}$. Since $\mathrm{Fil}_v^+(T)$ and $\mathrm{Fil}_v^+(A)$ are orthogonal to each other with respect to \eqref{pairing2}, $\mathrm{Fil}_v^+(\mathbf{T})$ and $\mathrm{Fil}_v^+(\mathbf{A})$ are orthogonal to each other with respect to the pairing $(\ ,\ )_\infty$.

We will denote by $\mathcal{F}_\Lambda$ the following Selmer structure of Greenberg type on $\mathbf{T}$ over $K$:
$$H^1_{\mathcal{F}_\Lambda}(K_v,\mathbf{T})=
\begin{cases}
\im(H^1(K_v,\mathrm{Fil}_v^+(\mathbf{T}))\rightarrow H^1(K_v,\mathbf{T})) & \text{if}\ v\mid p\\
H^1_\mathrm{unr}(K_v,\mathbf{T}) & \text{if}\ v\nmid p
\end{cases}$$
The corresponding Selmer group is denoted by $H^1_{\mathcal{F}_\Lambda}(K,\mathbf{T})$. In the same way, using $\mathrm{Fil}_v^+(\mathbf{A})$ we can define a Selmer structure, that we continue to denote by $\mathcal{F}_\Lambda$, on $\mathbf{A}$ over $K$ and consequently we have the Selmer group $H^1_{\mathcal{F}_\Lambda}(K,\mathbf{A})$.


\subsection{Comparison of Selmer groups}
Since we are assuming the $f$ is ordinary at $p$, by \cite[(23)]{Howard-Var}, for any number field $L$ there is a natural identification
$$H^1_\mathrm{BK}(L,V)=H^1_\mathrm{Gr}(L,V),$$
and the same is true if we replace $V$ with $T$.

From this identification, as observed in \cite[\S 5.1]{LV-Kyoto} using Shapiro's lemma one gets the isomorphism
\begin{equation}\label{comparison T}
\hat H_{\mathrm{BK}}^1(K_\infty, T)\simeq H^1_{\mathcal{F}_\Lambda}(K,\mathbf{T}).
\end{equation}
Recall that we defined $\mathcal{X}_\infty=H^1_{\mathrm{BK}}(K_\infty,A)^\vee$. Using the so-called control theorems (see \cite[Proposition 3.4]{LV-Kyoto}), one can also show that there exists a pseudo-isomorphism of compact $\Lambda$-modules
\begin{equation}\label{comparison Iw}
\mathcal{X}_\infty \sim H^1_{\mathcal{F}_\Lambda}(K,\mathbf{A})^\vee.
\end{equation}

\subsection{Bounding Selmer groups}

Let $T$ be a finitely generated module over a coefficient ring $R$ of residue characteristic $p$ equipped with a continuous action of $G_K$. Following the notation in \cite{Howard}, let $\mathcal{L}_0=\mathcal{L}_0(T)$ be the set of degree two primes of $K$ which do not divide $p$ and do not belong to the finite set of primes at which $T$ is ramified. A triple $(T,\mathcal{F},\mathcal{L})$ where $\mathcal{F}$ is a Selmer structure on $T$ and $\mathcal{L}\subseteq \mathcal{L}_0$ is a subset disjoint from $\Sigma(\mathcal{F})$ is called a \emph{Selmer triple}.

We introduce the following notation.
\begin{enumerate}[(a)]
\item For each $\lambda\in \mathcal{L}_0$, let $\ell$ be the rational prime below it and define $I_\ell$ to be the smallest ideal of $R$ containing $\ell +1$ and such that $\Frob_{\lambda}$ acts trivially on $T/I_\ell T$.\\
\item For every $k\in\Z_{\geq 1}$ define $\mathcal{L}_k=\mathcal{L}_k(T)=\{\ell\in\mathcal{L}_0\ |\ I_\ell\subseteq p^k\Z_p\}$.\\
\item For $\lambda\in\mathcal{L}_0$ as above, set $G_\ell:=k_{\lambda}^\times/k_\ell^\times$ where $k_\lambda$ e $k_\ell$ are the residue fields of $\lambda$ and $\ell$, respectively.\\
\item Let $\mathcal{N}_0$ denote the set of square-free products of rational primes lying below the elements of $\mathcal{L}_0$. For $n\in \mathcal{N}_0$, define 
$$I_n:=\sum_{\ell\mid n}I_\ell \subseteq R \quad\text{and}\quad G_n:=\bigotimes_{\ell\mid n}G_\ell.$$
By convention $1\in \mathcal{N}_0$, $I_1=0$ and $G_1= \Z$.
\end{enumerate}
Let $\mathcal{N}(\mathcal{L})$ be the set of square-free products of rational primes lying below the primes of $\mathcal{L}$. Given $c\in \mathcal{N}(\mathcal{L})$, we define a new Selmer triple $(T,\mathcal{F} (c),\mathcal{L}(c))$ by taking $\Sigma(\mathcal{F} (c)):=\Sigma(\mathcal{F})\cup \{v\ \text{prime of}\ K: v\mid c\}$, $\mathcal{L}(c):=\mathcal{L}\setminus \{v\ \text{prime of}\ K: v\mid c\}$, and
$$H^1_{\mathcal{F}(c)}(K_v,T):=
\begin{cases}
H^1_{\mathcal{F}}(K_v,T) & \text{if}\ v\nmid c\\
H^1_{\mathrm{tr}}(K_v,T) & \text{if}\ v\mid c
\end{cases}$$
where $H^1_{\mathrm{tr}}(K_v,T)$ denotes the \emph{transverse} subgroup of $H^1(K_v,T)$, that is the kernel
$$H^1_{\mathrm{tr}}(K_v,T):=\ker(H^1(K_v,T)\longrightarrow H^1(K_v^{(p)},T))$$
where $K_v^{(p)}$ denotes the maximal totally tamely ramified abelian $p$-extension of $K_v$ given by the maximal $p$-subextension of $K[c]_v/K_v$. Note that $v$ splits completely in the Hilbert class field of $K$, and $p$ is prime to $c$.

In a general setting, if $v$ is a prime of $K$ which does not divide $p$ and $T$ is unramified at $v$, then the unramified subgroup $H^1_{\mathrm{unr}}(K_v,T)$ is also referred to as the \emph{finite} part of $H^1(K_v,T)$ and written as $H^1_{\mathrm{f}}(K_v,T)$. The quotient of $H^1(K_v,T)$ by $H^1_{\mathrm{f}}(K_v,T)$ is the so-called \emph{singular} part and it is denoted by $H^1_{\mathrm{s}}(K_v,T)$. It turns out that, if $| k_v^\times |\cdot T=0$ then there are canonical isomorphisms
$$H^1_\mathrm{f}(K_v,T)\simeq T/(\Frob_v -1)T, \qquad H^1_\mathrm{s}(K_v,T)\otimes k_v^\times\simeq T^{\Frob_v=1},$$
see \cite[Proposition 1.1.7]{Howard}.
It follows that, if in addition a decomposition group $G_{K_v}$ at $v$ acts trivially on $T$ then one can define the \emph{finite-singular comparison map} as the isomorphism
\begin{equation}\label{fs iso}
\phi_v^{\mathrm{fs}}\colon H^1_\mathrm{f}(K_v,T)\simeq T\simeq H^1_\mathrm{s}(K_v,T)\otimes k_v^\times.
\end{equation}
Moreover, in this case the transverse submodule $H^1_\mathrm{tr}(K_v,T)$ projects isomorphically onto $H^1_\mathrm{s}(K_v,T)$ giving a splitting
$$H^1(K_v,T)=H^1_\mathrm{f}(K_v,T)\oplus H^1_\mathrm{tr}(K_v,T).$$

Let $\mathcal{L}_0=\mathcal{L}_0(\mathbf{T})$. For any $n\ell\in\mathcal{N}_0$, take $T=\mathbf{T}/I_{n\ell}\mathbf{T}$ in the above definitions, where $\mathbf{T}$ is the $\Lambda$-module defined in \ref{Greenberg Lamb-adic}. Then, from \eqref{fs iso} we get the isomorphism
$$\phi_\ell^\mathrm{fs}\colon H^1_\mathrm{f}(K_\ell,\mathbf{T}/I_{n\ell}\mathbf{T})\simeq H^1_\mathrm{s}(K_\ell,\mathbf{T}/I_{n\ell}\mathbf{T})\otimes G_\ell$$
where we have consistently interchanged $\ell\in\mathcal{L}_0$ with the prime $v$ of $K$ above it.

Let $(\mathbf{T},\mathcal{F}_\Lambda,\mathcal{L})$ be a Selmer triple in which $\mathbf{T}$ and $\mathcal{F}_\Lambda$ are defined as in section \ref{Greenberg Lamb-adic}. For any $n\ell\in \mathcal{N}(\mathcal{L})$ with $\ell$ prime, there is a diagram
\begin{equation}\label{diag}
\xymatrix{
 & & H^1_{\mathcal{F}_\Lambda(n\ell)}(K,\mathbf{T}/I_{n\ell}\mathbf{T})\otimes G_{n\ell}\ar[d]^{\mathrm{loc}_\ell}\\
 H^1_{\mathcal{F}_\Lambda(n)}(K,\mathbf{T}/I_{n}\mathbf{T})\otimes G_n \ar[r]^{\mathrm{loc}_\ell} & H^1_{\mathrm{f}}(K_\ell,\mathbf{T}/I_{n\ell}\mathbf{T})\otimes G_n \ar[r]^{\phi_\ell^\mathrm{fs}\otimes 1} & H^1_{\mathrm{s}}(K_\ell,\mathbf{T}/I_{n\ell}\mathbf{T})\otimes G_{n\ell}
}
\end{equation}
\begin{definition}
A \emph{Kolyvagin system} $\bm{\kappa}$ for the Selmer triple $(\mathbf{T},\mathcal{F}_\Lambda,\mathcal{L})$ is a collection of cohomology classes 
$$\kappa_n\in H^1_{\mathcal{F}_\Lambda(n)}(K,\mathbf{T}/I_n \mathbf{T})\otimes G_n$$
for each $n\in \mathcal{N}(\mathcal{L})$, such that for any $n\ell \in \mathcal{N}(\mathcal{L})$ the images of $\kappa_n$ and $\kappa_{n\ell}$ in $H^1_\mathrm{s}(K_\ell,\mathbf{T}/I_{n\ell}\mathbf{T})\otimes G_{n\ell}$ under the maps $(\phi_\ell^\mathrm{fs}\otimes 1)\circ \mathrm{loc}_\ell$ and $\mathrm{loc}_\ell$ in \eqref{diag} agree.
\end{definition}

From now on, we assume that the following ``big image" condition holds
\begin{assumption}\label{big image}
$\{g\in\GL_2(\Z_p)\ :\ \det(g)\in (\Z_p^\times)^{k-1}\}\subseteq \im(G_\Q\overset{\rho_{f,\p}}{\longrightarrow} \Aut(T)\simeq \GL_2(\cO_\p)).$
\end{assumption}

The following structure theorem is a fundamental result for bounding Selmer groups over the anticyclotomic $\Z_p$-extension of $K$.
\begin{theorem}\label{structure}
Let $X:=H^1_{\mathcal{F}_\Lambda}(K,\mathbf{A})^\vee$. Suppose that for some $s\geq 1$ the Selmer triple $(\mathbf{T},\mathcal{F}_\Lambda,\mathcal{L}_s)$ admits a $\Lambda$-adic Kolyvagin system $\bm\kappa$ with $\kappa_1\neq 0$. Then
\begin{enumerate}
\item $H^1_{\mathcal{F}_\Lambda}(K,\mathbf{T})$ is a torsion free $\Lambda$-module of rank $1$;\\
\item there exist a torsion $\Lambda$-module $M$ such that $\mathrm{char} (M)=\mathrm{char}(M)^\iota$ and a pseudo-isomorphism
$$X\sim \Lambda\oplus M \oplus M;$$
\item $\mathrm{char}(M)$ divides $\mathrm{char}(H^1_{\mathcal{F}_\Lambda}(K,\mathbf{T})/\Lambda \kappa_1)$.
\end{enumerate}
\end{theorem}
\begin{proof}
See \cite[Theorem 3.5]{LV-Kyoto} and \cite[Theorem 2.2.10]{Howard}. Note that the first line in \cite[Theorem 3.5]{LV-Kyoto} is superfluous as that hypothesis is not necessary to prove the theorem.
\end{proof}
In light of Theorem \ref{structure}, it remains to construct a Kolyvagin system for the Selmer triple $(\mathbf{T},\mathcal{F}_\Lambda,\mathcal{L}_s)$ such that $\kappa_1\neq 0$. This is accomplished in the following section.

\section{A $\Lambda$-adic Kolyvagin system of generalized Heegner cycles in a quaternionic setting}
This section contains one of the main results of this paper, that is the construction of a $\Lambda$-adic Kolyvagin system of generalized Heegner classes coming from cycles on a variety fibered over a Shimura curve. This object is the main ingredient, together with Theorem \ref{structure}, to prove the one-sided divisibility relation in the anticyclotomic Iwasawa main conjecture for modular forms and an imaginary quadratic fields that satisfy only a relaxed Heegner hypothesis in the sense specified in the introduction, and hence that do not have Heegner cycles coming from a modular curve.

Recall from Section \ref{class-Paola} the class $z_{f,1,c_0 p^s}\in H^1(K[c_0 p^s],T)$ in \eqref{class}, for every positive integer $c_0$ coprime with $pN^+$ and $s\geq 0$. To simplify the notation, we will just write $z_{c_0 p^s}$ for such a class. By \cite[Theorem 3.1]{Nekovar-AJ}, the classes $z_{c_0 p^s}$ lie in the Bloch-Kato Selmer group $H^1_{\mathrm{BK}}(K[c_0 p^s],T)$.

Let $\mathcal{L}:=\mathcal{L}_1(\mathbf{T})$ such that $(\mathbf{T},\mathcal{F}_\Lambda,\mathcal{L})$ is a Selmer triple, and let $\mathcal{N}:=\mathcal{N}(\mathcal{L})$. For every $\ell\in\mathcal{N}$ write $\lambda$ for the unique prime of $K$ above $\ell$, and fix a prime $\bar\lambda$ of $\bar \Q$ above $\ell$.

For $n\in \mathcal{N}$ let $K_m[n]$ be the compositum of $K_m$ and $K[n]$, and let $K_\infty[n]$ be the union over all integers $m\geq 0$ of $K_m[n]$. The prime $\bar \lambda$ determines a prime $\lambda_{m,n}\in K_m[n]$; we denote by $K_m[n]_{\bar\lambda}$ the completion of $K_m[n]$ at $\lambda_{m,n}$.

For $n\in\mathcal{N}$ and $m\geq 0$ define classes $\alpha_m[n]\in H^1_\mathrm{BK}(K_m[n],T)$ by
$$\alpha_m[n]:=\cores_{K[np^{m+1}]/K_m[n]}(z_{n p^{m+1}}).$$
Let $\mathcal{G}(n):=\Gal(K[n]/K)$ and $\sigma_\wp$, $\sigma_{\bar\wp}$ be the Frobenius elements at the two primes $\wp$, $\bar\wp$ of $K$ above $p$. Define $\gamma_m\in \cO_\p[\mathcal{G}(n)]$ by the recursive formula
\begin{equation}\label{recursive-def}
\gamma_m:= a_p\gamma_{m-1}-p^{k-1}\gamma_{m-2}
\end{equation}
for all $m\geq 3$ with
\begin{align*}
\gamma_1 &:= a_p-p^{\frac{k-2}{2}}(\sigma_\wp+\sigma_{\bar\wp}),\\
\gamma_2 &:= a_p\gamma_1-p^{k-2}(p-1)=a_p^2-p^{\frac{k-2}{2}}a_p(\sigma_\wp+\sigma_{\bar\wp})-p^{k-2}(p-1).
\end{align*}
The following relations between the classes $\alpha_m[n]$ hold.

\begin{lemma}\label{relations}
For all $n\geq 1$,
\begin{enumerate}
\item $\cores_{K_{m+1}[n]/K_m[n]}(\alpha_{m+1}[n])=a_p \alpha_m[n]-p^{k-2}\res_{K_{m-1}[n]/K_m[n]}(\alpha_{m-1}[n])$, for all $m\geq 1$;\\
\item $\cores_{K_m[n]/K[n]}(\alpha_m[n])=\gamma_{m+1} z_n$, for all $m\geq 0$;\\
\item $\cores_{K_m[n\ell]/K_m[n]}(\alpha_m[n\ell])=a_\ell \alpha_m[n]$ for all $\ell\nmid n$ inert in $K$ and all $m\geq 0$.
\end{enumerate}
\end{lemma}
\begin{proof}
(1) By definition one has
\begin{equation*}
\begin{split}
\cores_{K_{m+1}[n]/K_m[n]}(\alpha_{m+1}[n]) & = \cores_{K[np^{m+2}]/K_m[n]}(z_{np^{m+2}})\\
& = \cores_{K[np^{m+1}]/K_m[n]}\cores_{K[np^{m+2}]/K[np^{m+1}]}(z_{np^{m+2}}).
\end{split}
\end{equation*}
Using the first part of \cite[Proposition 5.3]{Magrone}, we have $\cores_{K[np^{m+2}]/K[np^{m+1}]}(z_{np^{m+2}})=a_p z_{np^{m+1}}-p^{k-2} z_{np^m}$, and hence
$$\cores_{K_{m+1}[n]/K_m[n]}(\alpha_{m+1}[n])= a_p \alpha_m[n]-p^{k-2}\cores_{K[np^{m+1}]/K_m[n]}(z_{np^m}).$$
Finally,
\begin{equation*}
\begin{split}
\cores_{K[np^{m+1}]/K_m[n]}(z_{np^m})&=\res_{K_{m-1}[n]/K_m[n]}(\cores_{K[np^m]/K_{m-1}[n]}( z_{np^m}))\\
&=\res_{K_{m-1}[n]/K_m[n]}(\alpha_{m-1}[n]).
\end{split}
\end{equation*}

For $m\geq 2$, (2) can be proved by induction on $m$ using the first part of \cite[Proposition 5.3]{Magrone} and the first two steps, $m=0$ and $m=1$, that can be checked directly. For $m=0$ on the left hand side one just has $\alpha_0[n]$ which by definition is $\cores_{K[np]/K[n]}(z_{np})$. Then the claim follows from the equality
$$T_p z_n=\cores_{K[np]/K[n]}(z_{np})+p^{\frac{k-2}{2}}\sigma_\wp z_n+p^{\frac{k-2}{2}}\sigma_{\bar\wp} z_n,$$
which is the analogue for generalized Heegner classes coming from cycles on Shimura curves of the well known formula for Heegner points that one can find in \cite[\S 3.1 Proposition 1]{Perrin-Riou} (note that we are assuming $\# \cO_K^\times =2$ and hence $\delta=1$ in \emph{op.cit.}). For $m=1$, one has 
$$\cores_{K_1[n]/K[n]}(\alpha_1[n])=\cores_{K[np^2]/K[n]}(z_{np^2})$$
and the claim follows as in the proof of part (1) and observing that $[K[np]:K[n]]=p-1$.

(3) is an easy computation after applying the formula
$$T_\ell z_c=\cores_{K[c\ell]/K[c]}(z_{c\ell})$$
for all $c$ with $\ell\nmid c$ inert in $K$, proved in \cite[Proposition 5.3]{Magrone}.
\end{proof}

Let $\alpha$ be the $p$-adic unit root of $X^2-a_p X+p^{k-1}$, and define $\Phi\in \cO_\p[\mathcal{G}(n)]$ as the element
\begin{equation}\label{phi}
\Phi=\left(1-p^{(k-2)/2}\frac{\sigma_\wp}{\alpha}\right)\left(1-p^{(k-2)/2}\frac{\sigma_{\bar\wp}}{\alpha}\right).
\end{equation}

\begin{lemma}\label{recursive}
For all $n\geq 2$ the elements $\gamma_n\in\cO_\p[\mathcal{G}(n)]$ can be written as
\[ \gamma_n=q_n\Phi+\frac{p^{(n-1)(k-1)}}{\alpha^{n-1}}\gamma_1,\]
where $q_{n+1}\equiv a_p q_n\pmod p$.
\end{lemma}
\begin{proof}
We proceed by induction as in \cite[Lemma 4]{Perrin-Riou}. Take $q_n$ to be the elements $$\sum_{j=0}^{n-2}\alpha^{n-2j}p^{j(k-1)}.$$ 
By a direct computation, for $n=2$ one has $\gamma_2=\alpha^2\Phi+\frac{p^{k-1}}{\alpha} \gamma_1$. Now we look at the case $n=3$ using the case $n=2$. By the recursive formula \eqref{recursive-def}, we know that
\[ \gamma_3=a_p\gamma_2-p^{k-1}\gamma_1, \]
and replacing the previous expression for $\gamma_2$ we obtain
\[ \gamma_3=a_p\alpha^2\Phi+a_p\frac{p^{k-1}}{\alpha}\gamma_1-p^{k-1}\gamma_1. \]
Using the fact that $a_p=\alpha+\frac{p^{k-1}}{\alpha}$ we obtain the desired formula. Now fix $n\geq 3$ and assume that the statement of the lemma is true for all $t$ with $2\leq t\leq n$. We prove the formula for $\gamma_{n+1}$:
\begin{align*}
\gamma_{n+1} &=a_p\gamma_n-p^{k-1}\gamma_{n-1}\\
                        &=\left(\alpha+\frac{p^{k-1}}{\alpha}\right)\left(q_n\Phi+\frac{p^{(n-1)(k-1)}}{\alpha^{n-1}}\gamma_1\right)-p^{k-1}\left(q_{n-1}\Phi+\frac{p^{(n-2)(k-1)}}{\alpha^{n-2}}\gamma_1\right)\\
                        &=\left(\sum_{t=0}^{n-1}\alpha^{n+1-2t}p^{t(k-1)}\right)\Phi+\frac{p^{n(k-1)}}{\alpha^n}\gamma_1
\end{align*}
as desired.
\end{proof}
\begin{lemma}\label{perrin}
If $M$ is any finitely generated $\cO_\p[\mathcal{G}(n)]$-module, the intersection of $\gamma_m M$ for $m\geq 1$ is equal to $\Phi M$.
\end{lemma}
\begin{proof}
This follows from Lemma \ref{recursive}, as for Corollaire 5 of \cite[\S 3.3]{Perrin-Riou}.
\end{proof}

For $n\in\mathcal{N}$and $m\geq 0$ denote by $H_m[n]$ the $\cO_\p[\Gal(K_m[n]/K)]$-submodule of $H^1_{\mathrm{BK}}(K_m[n],T)$ generated by the restrictions of the classes $z_n$ and $\alpha_j[n]$ for all $j\leq m$. Part (1) of Lemma \ref{relations} implies that the corestriction from $K_{m+1}[n]$ to $K_m[n]$ sends $H_{m+1}[n]$ to $H_m[n]$. Hence we can define a $\Lambda[\mathcal{G}(n)]$-module
\begin{equation}\label{Iw module}
H_\infty[n]:=\varprojlim_m H_m[n]\subseteq \varprojlim_m H^1_\mathrm{BK}(K_m[n],T).
\end{equation}
Note that $\varprojlim_m H^1_\mathrm{BK}(K_m[n],T)\subseteq \varprojlim_m H^1(K_m[n],T)=:H^1(K_\infty[n],T)$ and, as observed in \cite[\S 3.3 eq. (4)]{LV-Kyoto}, $H^1(K_\infty[n],T)\simeq H^1(K[n],\mathbf{T})$. Therefore, in the next we will sometimes view $H_\infty[n]$ as contained in $H^1(K[n],\mathbf{T})$.

\begin{proposition}\label{family}
There exists a family
$$\{\beta[n]=(\beta_m[n])_{m\geq 0}\in H_\infty[n]\}_{n\in\mathcal{N}}$$
such that $\beta_0[n]=\Phi z_n$, and for any $n\ell\in\mathcal{N}$
$$\cores_{K_\infty[n\ell]/K_\infty[n]}(\beta[n\ell])=a_\ell \beta[n].$$
\end{proposition}
\begin{proof}
The proof runs as that of \cite[Lemma 2.3.3]{Howard}. Fix $n\in\mathcal{N}$. For $m\geq 0$ one takes $\tilde{H}_m$ to be the quotient of the free $\cO_\p[\Gal(K_m[n]/K)]$-module generated by $\{x, x_j\ |\ 0\leq j\leq m\}$ modulo the following relations:
\begin{itemize}
\item $\sigma(x)=x$ for all $\sigma\in\Gal(K_m[n]/K[n])$;\\
\item $\sigma(x_j)=x_j$ for $\sigma\in\Gal(K_m[n]/K_j[n])$;\\
\item $\mathrm{Tr}_{K_j[n]/K_{j-1}[n]}(x_j)=a_p x_{j-1}-p^{k-2} x_{j-2}$ for $j\geq 2$;\\
\item $\mathrm{Tr}_{K_1[n]/K[n]}(x_1)=\gamma_2 x$ and $x_0=\gamma_1 x$.
\end{itemize}
It follows that for each $j\leq k$
$$\mathrm{Tr}_{K_j[n]/K[n]}(x_j)=\gamma_{j+1}x.$$
There is a natural inclusion $\tilde H_{m}\rightarrow \tilde H_{m+1}$ and a natural trace map $\tilde H_{m+1}\rightarrow \tilde H_m$. By lemma \ref{perrin} and the above equation, we have that $\Phi x\in\tilde H_0$ is a trace from every $\tilde H_m$. This implies that we can choose an element $y\in\tilde H_\infty:=\varprojlim_m \tilde H_m$ which is a lift of $\Phi x$. Then for any $t$ such that $t\mid n$, we define $\beta[t]\in H_\infty[t]$ to be the the image of $y$ under the the following map: $\phi(t)\colon \tilde H_\infty\rightarrow H_\infty[t]$ sending $x_m$ to $\alpha_m[t]$ and $x$ to $z_t$. As shown in the proof of \cite[Lemma 2.3.3]{Howard} and in the proof of \cite[Proposition 4.9]{LV-Kyoto} these families have the desired properties and they are enough to show our claim.
\end{proof}

For $n\in\mathcal{N}$ recall that $\mathcal{G}(n)=\Gal(K[n]/K)$ and let $G(n)=\prod_{\ell\mid n} G_\ell$ \footnote{Do not confuse with $G_n=\otimes_{\ell\mid n} G_\ell$}. By convention, put $G(1)=1$. Then, for $m$ dividing $n$ we have
$$\Gal(K[n]/K[m])\simeq \prod_{\ell\mid (n/m)}G_\ell = G(n/m).$$
So in particular $G(n)$ is a subgroup of $\mathcal{G}(n)$. For each prime $\ell\in \mathcal{N}$ fix a generator $\sigma_\ell $ of $G_\ell$ and define Kolyvagin's derivative operator as
$$D_\ell := \sum_{i=1}^\ell i \sigma_\ell^i\ \in \Z[G_\ell].$$
For each $n\in \mathcal{N}$ put $D_n:=\prod_{\ell\mid n} D_\ell \in \Z[G(n)]$, and for $n=1$ define $D_1$ to be the identity operator. Now fix a set $S$ of representatives of $G(n)$ in $\mathcal{G}(n)$, and let
$$\tilde \kappa_n:=\sum_{s\in S} s D_n \beta[n] \in H_\infty [n]. $$
Note that $sD_n \beta[n]$ makes sense since $\beta[n]\in H_\infty [n]$ and $H_\infty [n]$ is a $\Lambda[\mathcal{G}(n)]$-module.

As observed just above Proposition \ref{family}, $H_\infty[n]\subseteq H^1(K[n],\mathbf{T})$. For each $n\in \mathcal{N}$ we can consider the image of $\tilde\kappa_n$ in $H^1(K[n],\mathbf{T}/I_n \mathbf{T})$.

\begin{lemma}
The image of $\tilde\kappa_n$ in $H^1(K[n],\mathbf{T}/I_n \mathbf{T})$ is fixed by $\mathcal{G}(n)$.
\end{lemma}
\begin{proof}
It is enough to prove that for each $\ell\mid n$, $(\sigma_\ell -1)D_n \beta[n]=0$ in $H^1(K[n],\mathbf{T}/I_n\mathbf{T})$. Observe that $(\sigma_\ell -1)D_\ell$ acting on $H^1(K[n],\mathbf{T})$ is equal to
\begin{equation}
\begin{split}
(\sigma_\ell -1)D_\ell &= \sum_{i=1}^\ell i\sigma_\ell^{i+1} -\sum_{i=1}^\ell i\sigma_\ell^i\\
&= -\sigma_\ell-\sigma_\ell^2-\cdots -\sigma_\ell^\ell+\ell\sigma_\ell^{\ell+1}\\
&=-(\cores_{K[n]/K[n/\ell]}-1)+\ell\\
&=\ell+1-\cores_{K[n]/K[n/\ell]}.
\end{split}
\end{equation}
Therefore
\begin{equation}
\begin{split}
(\sigma_\ell-1)D_n\beta[n]&= (\sigma_\ell-1)D_\ell D_{n/\ell}\beta[n]\\
&=D_{n/\ell}(-\cores_{K[n]/K[n/\ell]})\beta[n]\quad \text{in}\ H^1(K[n],\mathbf{T}/I_n\mathbf{T})\\
&=-a_\ell D_{n/\ell}\beta[n/\ell]=0\quad \text{in}\ H^1(K[n],\mathbf{T}/I_n\mathbf{T}),
\end{split}
\end{equation}
where we used the fact that $\ell+1$ and $a_\ell$ belong to $I_\ell\subseteq I_n$.
\end{proof}
As observed in \cite[\S 4.4]{LV-Kyoto} under our hypothesis the restriction map induces an isomorphism
\begin{equation}\label{isomorphism}
H^1(K,\mathbf{T}/I_n\mathbf{T})\overset{\simeq}{\longrightarrow} H^1(K[n],\mathbf{T}/I_n\mathbf{T})^{\mathcal{G}(n)}.
\end{equation}
Define $\kappa_n\in H^1(K,\mathbf{T}/I_n\mathbf{T})$ to be the preimage of $\tilde\kappa_n$ under this isomorphism.

\begin{lemma}
For every $n\in\mathcal{N}$, $\kappa_n\in H^1_{\mathcal{F}_\Lambda(n)}(K,\mathbf{T}/I_n\mathbf{T})$.
\end{lemma}
\begin{proof}
By definition, one has to show that the localization of $\kappa_n$ at every prime $v$ such that $v\mid n$ lies in the transverse subgroup, and that the localization of $\kappa_n$ at every prime $v$ such that $v\nmid n$ lies in the image of
$$H^1_{\mathcal{F}_\Lambda}(K_v,\mathbf{T})\longrightarrow H^1(K_v,\mathbf{T}/I_n\mathbf{T}).$$
In both cases the proof runs as in the proof of \cite[Lemma 4.12]{LV-Kyoto}, that in turn is inspired by the proof of \cite[Lemma 2.3.4]{Howard}.
\end{proof}

\begin{theorem}\label{Koly system}
Assume that $\tilde\kappa_1$ is non-torsion. Then, there exists a Kolyvagin system $\bm\kappa^\mathrm{Heeg}$ for the Selmer triple $(\mathbf{T},\mathcal{F}_\Lambda,\mathcal{L})$ with $\kappa_1^\mathrm{Heeg}\neq 0$.
\end{theorem}

\begin{proof}
Define $\bm\kappa^\mathrm{Heeg}=\{\kappa_n^\mathrm{Heeg}\}$ as
$$\kappa_n^\mathrm{Heeg}:=\bigg(\prod_{\ell\mid n} u_\ell^{-1}\bigg)\cdot \kappa_n\otimes \bigg(\bigotimes_{\ell\mid n}\sigma_\ell\bigg),$$
where, for every $\ell\in\mathcal{L}$, $u_\ell$ is the $p$-adic unit satisfying the relation
\begin{equation}\label{Koly}
\mathrm{loc}_\ell(\kappa_{n\ell})\otimes \sigma_\ell=u_\ell\cdot \phi_\ell^{\mathrm{fs}}(\mathrm{loc}_\ell(\kappa_n)).
\end{equation}
The existence of such a unit is given by \cite[Proposition 5.7]{Elias-deVeraPiq}. Note that the definition of $\phi_\ell^{\mathrm{fs}}$ (denoted by $\phi_{\lambda,s'}$ in \emph{loc. cit.}) is slightly different from ours as it is given by composing the evaluation at $\Frob_\ell$ of elements in $H^1_\mathrm{f}(K_v,\mathbf{T}/I_{n\ell}\mathbf{T})$ with the inverse of the evaluation at the fixed generator $\sigma_\ell$ of $G_\ell$. Hence it produces an element in $H^1_\mathrm{s}(K_v,\mathbf{T}/I_{n\ell}\mathbf{T})$ instead of an element in $H^1_\mathrm{s}(K_v,\mathbf{T}/I_{n\ell}\mathbf{T})\otimes G_\ell$ as happens in our case.

Using \eqref{Koly} it is easy to see that $\bm\kappa^\mathrm{Heeg}=\{\kappa_n^\mathrm{Heeg}\}$ is a Kolyvagin system for $(\mathbf{T},\mathcal{F}_\Lambda,\mathcal{L})$.

It remains to prove that $\kappa_1^\mathrm{Heeg}\neq 0$.

For every $m\geq 0$ let $H_m$ be the $\Lambda_m=\cO_\p[\Gal(K_m/K)]$-submodule of $H^1_\mathrm{BK}(K_m,T)$ generated by $\res_{K_m/K}(\cores_{K[1]/K}(z_1))$ and $\cores_{K_m[1]/K_m}(\res_{K_m[1]/K_j[1]}(\alpha_j[1]))$ for all $j\leq m$. As in the definition \eqref{Iw module}, Lemma \ref{relations} implies that corestriction from $K_{m+1}$ to $K_m$ sends $H_{m+1}$ to $H_m$, and we can define the compact $\Lambda$-module
$$H_\infty:=\varprojlim_m H_m\subseteq \varprojlim_m H^1_{\mathrm{BK}}(K_m,T)=\hat H^1_\mathrm{BK}(K_\infty,T)\simeq H_{\mathcal{F}_\Lambda}^1(K,\mathbf{T})$$
where the last isomorphism follows from \eqref{comparison T}.

Recall that by convention $I_1=0$; hence we have that $\kappa_1^\mathrm{Heeg}=u_1^{-1}\cdot \kappa_1$ and $\kappa_1\in H_{\mathcal{F}_\Lambda(1)}^1(K,\mathbf{T})=H^1_{\mathcal{F}_\Lambda}(K,\mathbf{T})$ is the image of 
$$\tilde\kappa_1=\sum_{\sigma\in \Gal(H/K)}\sigma\beta[1]\in H_\infty[1]\subseteq H^1(H,\mathbf{T})^{\Gal(H/K)}$$
under the isomorphism \eqref{isomorphism}, where we recall that $H=K[1]$ is the Hilbert class field of $K$. It follows that we can view $\tilde \kappa_1$ as an element of $H_\infty$. The Theorem \ref{rank1} below completes the proof.
\end{proof}

\begin{theorem}\label{rank1}
Assume that $\tilde\kappa_1$ is non-torsion. Then, the $\Lambda$-module $H_\infty$ is free of rank $1$, generated by $\tilde \kappa_1$.
\end{theorem}
\begin{proof}
Assumption \ref{assumption} implies that the sign of the functional equation of the complex $L$-function of $f$ over $K$ is $-1$. Hence the analogue of \cite[Theorem 6.1 (2)]{CH} for our generalized Heegner classes guarantees the non-torsionness of $\cores_{K_m[1]/K_m}(\alpha_m[1])$ for $m$ sufficiently large. 

In fact, as for \cite[Theorem 6.1]{CH}, such a result follows from \cite[Theorem 6.4]{Magrone}, which relates our Heegner classes to special values of a certain $p$-adic $L$-function, and \cite[Theorem 4.6]{Magrone}, which gives the non-vanishing of this $p$-adic $L$-function and whose proof is announced to appear in a forthcoming work. We assume here that the proof of \cite[Theorem 4.6]{Magrone} is done, and hence that the non-torsionness of $\cores_{K_m[1]/K_m}(\alpha_m[1])$ for $m\gg 0$ holds. From this one gets that $H_\infty$ is free of rank $1$ over $\Lambda$ by mimicking the proof of \cite[\S 3.4, Proposition 10]{Perrin-Riou} where an analogous result is obtained for Heegner points.

Therefore it remains to show that $H_\infty$ is generated by $\tilde \kappa_1$. The proof of this last claim runs exactly as the proof of \cite[Theorem 4.18]{LV-Kyoto} hence we omit it.
\end{proof}

\subsection{Proof of the main theorem}
The proof of our main result \ref{main} follows by combining Theorem \ref{structure}, Theorem \ref{Koly system} and the comparison between Selmer groups \eqref{comparison T} and \eqref{comparison Iw}.

\section{$p$-adic $L$-functions and Selmer groups}
There exists another formulation of the Iwasawa main conjecture which involves $p$-adic $L$-functions. As proved by Castella and Hsieh in the classical Heegner setting, and by Magrone under the relaxed Heegner hypothesis, there exists a $p$-adic $L$-function $\mathscr{L}_\p(f)$ whose values at anticyclotomic characters of conductor $p^n$ are related to the generalized Heegner class $z_{p^n}$ of conductor $p^n$ through a $p$-adic Gross-Zagier type formula. On the other hand, we saw in the previous sections that the classes $z_{p^n}$ play a role in the structure of the Selmer group $\mathcal{X}_\infty$. Hence it is natural to expect a formulation of the main conjecture in terms of the $p$-adic $L$-function $\mathscr{L}_\p(f)$.

Let $H^1_\mathrm{Iw}(K[p^\infty],T)$ be the Iwasawa cohomology group
$$H^1_\mathrm{Iw}(K[p^\infty],T):=\varprojlim _n H^1(\Gal(K^{Np}/K[p^n]),T),$$
where $K^{Np}$ is the maximal extension of $K$ unramified outside the primes above $Np$.
Let $X^2-a_p X+p^{k-1}$ be the Hecke polynomial of $f$ at $p$ and $\alpha$ be its $p$-adic unit root. Define the $\alpha$-stabilized Heegner class $z_{p^n,\alpha}$ as
$$z_{p^n,\alpha}:=z_{p^n}-\frac{p^{k-2}}{\alpha}z_{p^{n-1}}$$
for all $n\geq 1$, and $z_{1,\alpha}=\Phi z_1$ where $\Phi$ is the element defined in \eqref{phi}. As proved in \cite[\S 7.1.2]{Magrone}, the classes $\alpha^{-n}z_{p^n,\alpha}$ form a compatible system with respect to the corestriction maps (cf. $u_c=\# \cO_c^\times/2$ is $1$ under our running hypothesis). Thus we can consider the element
$$\boldsymbol{z}_{1,\alpha}:=\varprojlim_n \alpha^{-n}z_{p^n,\alpha}$$
in $H^1_\mathrm{Iw}(K[p^\infty],T)$. By Shapiro's lemma, there is an isomorphism $H^1_\mathrm{Iw}(K[p^\infty],T)\simeq H^1(K[1],T\otimes \cO_\p\llbracket \Gal(K[p^\infty]/K[1])\rrbracket)$. Let $\mathcal{G}_{p^\infty}:=\Gal(K[p^\infty]/K)$ be the Galois group of the tower of the ring class fields of $p$-power conductor over $K$. Recall from section \ref{Zp-ext} that $\mathcal{G}_{p^\infty}\simeq \Gamma_\infty\times \Delta$ where $\Delta\simeq \Gal(K[p]/K)$ has order prime to $p$. Hence we can view $\boldsymbol{z}_{1,\alpha}$ as an element of $H^1(K[1],T\otimes \cO_\p\llbracket \mathcal{G}_{p^\infty}\rrbracket)$ via the map induced in cohomology by the map $\cO_\p\llbracket \Gal(K[p^\infty]/K[1])\rrbracket\rightarrow \cO_\p\llbracket \mathcal{G}_{p^\infty}\rrbracket\simeq \cO_\p\llbracket \Gamma_\infty\times \Delta\rrbracket$ of $\Gal(\bar K/K[1])$-modules (see also \cite[\S 5.3]{CH}).

Define
$$\boldsymbol{z}:=\cores_{K[1]/K}(\boldsymbol{z}_{1,\alpha})\in H^1(K,T\otimes \cO_\p\llbracket \mathcal{G}_{p^\infty}\rrbracket).$$

Let $\varphi$ be the map
$$\varphi\colon H^1(K,T\otimes \cO_\p\llbracket \mathcal{G}_{p^\infty}\rrbracket)\longrightarrow H^1(K,\mathbf{T})$$
induced by the natural projection $\cO_\p\llbracket \mathcal{G}_{p^\infty}\rrbracket \rightarrow \Lambda$ (recall that $\mathbf{T}=T\otimes_{\cO_\p}\Lambda$).
\begin{lemma}\label{key-eq}
We have that
$$\varphi(\boldsymbol{z})=\kappa_1.$$
\end{lemma}
\begin{proof}
Recall that, by definition, $\kappa_1\in H^1_{\mathcal{F}_\Lambda}(K,\mathbf{T})$ is the image of
\[\tilde \kappa_1=\sum_{\sigma\in\Gal(K[1]/K)}\sigma\beta[1] \in H^1(K[1],\mathbf{T})^{\Gal(K[1]/K)}\]
under the isomorphism $H^1(K[1],\mathbf{T})^{\Gal(K[1]/K)}\simeq H^1(K,\mathbf{T})$.\\
In other words, $\kappa_1=\cores_{K[1]/K}(\beta[1])$. Hence, it is enough to show that $\varphi(\boldsymbol{z}_{1,\alpha})=\beta[1]$ (here we are considering the restriction of $\varphi$ to $H^1(K[1],T\otimes \cO_\p\llbracket \Gal(K[p^\infty]/K[1])\rrbracket)$).

Now, $\boldsymbol{z}_{1,\alpha}\in H^1_\mathrm{Iw}(K[p^\infty],T)\subseteq \varprojlim_m H^1(K[p^m],T)$ and, in terms of inverse limits, $\varphi$ is just the corestriction map from $H^1(K[p^m],T)$ to $H^1(K_{m-1}[1],T)$ for each $m\geq 1$. Finally, looking at the definition of $\beta[1]=(\beta_m[1])$ in Proposition \ref{family} one can see that $\beta_m[1]=\cores_{K[p^{m+1}]/K_m[1]}(\alpha^{-(m+1)}z_{p^{m+1},\alpha})$ which is our claim.
\end{proof}

\subsection{$p$-adic $L$-functions}
Fix an anticyclotomic Hecke character of infinity type $(k/2,-k/2)$ and conductor $\cO_K$ that we denote by $\psi$. In \cite[Definition 4.3]{Magrone}, Magrone defined a $p$-adic $L$-function $\mathscr{L}_{\p,\psi}(f)$ attached to $f$ and $\psi$ which should interpolate the algebraic part of the value at the critical point $s=k/2$ of the complex $L$-function of $f$ twisted by the Hecke characters $\psi\phi$, as $\phi$ varies in the Hecke characters of infinity type $(n,-n)$ with $n\geq 0$ and $p$-power conductor (see \cite[Theorem 4.6]{Magrone}). Note that as $\phi$ varies in such a space, $\psi\phi$ varies on the Hecke characters of infinity type $(j,-j)$ with $j\geq k/2$ (and $p$-power conductor) which is exactly the region of interpolation of the $p$-adic $L$-function of Bertolini-Darmon-Prasanna (\cite[\S 5.2]{BDP}) after taking into account that we are interpolating the values of complex $L$-functions at the critical point $s=k/2$, instead of $s=0$ as in \emph{op.cit.}.

Symmetrically the interpolation formula holds for Hecke characters $\psi\phi$ of infinity type $(j,-j)$ with $j\leq -k/2$ and conductor $p^n\cO_K$, and hence for $\phi$ with the same conductor and infinity type $(n,-n)$ such that $n\leq -k$.

Outside the region of interpolation, the values of the $p$-adic $L$-function $\mathscr{L}_{\p,\psi}(f)$ are related to the image under the Bloch-Kato logarithm map of the generalized Heegner classes defined above. More precisely, for every anticyclotomic Hecke character $\phi$ of infinity type $(n,-n)$ with $-k<n<0$ and conductor $p^t\cO_K$ with $t>0$ the following equality holds (see \cite[Theorem 6.4]{Magrone}):
\begin{equation}\label{Gross-Zag}
\frac{\mathscr{L}_{\p,\psi}(f)(\hat{\phi}^{-1})}{\Omega_p^*}=c_{\psi,\phi}\cdot\langle \log_{\p}(z_{p^t,\chi}), \omega_f\otimes\omega_{A,n}\rangle
\end{equation}
where
\begin{itemize}
\item $\hat\phi\colon \Gal(\overline{\Q}/K)\rightarrow \bar\Q_p^\times$ is the $p$-adic avatar of $\phi$ (see \cite[Definition 3.4]{CH})\\
\item $\Omega_p$ is a suitable $p$-adic period\\
\item $c_{\psi,\phi}$ is an explicit non-zero scalar factor depending on $\hat\psi$ and $\hat\phi$\\
\item $\log_\p:=\log \circ \mathrm{loc}_\p$ where $\log$ denotes the Bloch-Kato logarithm map (see \cite[\S 4.5]{CH}); it takes values in $\frac{\mathbf{D}_{\mathrm{dR},K_\p}(V)}{\mathrm{Fil}^0\mathbf{D}_{\mathrm{dR},K_\p}(V)}\simeq \mathrm{Fil}^{k-1}\epsilon_{Y_k}H^{2k-3}_\mathrm{dR}(Y_k/K_\p)^\vee$ where $\mathbf{D}_{\mathrm{dR},K_\p}$ is the Fontaine functor turning Galois representations in filtered modules, and $V=\epsilon_{Y_k}H^{2k-3}_\text{\'et}(Y_k\otimes \bar\Q,\Q_p(k-1))$\\
\item $\chi:=\hat \psi^{-1}\hat\phi$, where $\hat \psi$ is the $p$-adic avatar of $\psi$\\
\item $z_{p^t,\chi}\in H^1(K,T\otimes \chi)$ is the $\chi$-component of the generalized Heegner class $z_{f,p^t}=z_{p^t}$ in section \ref{class-Paola} as defined in \cite[\S 5.9]{Magrone}\\
\item $\omega_f\otimes\omega_{A,n}$ is a suitable differential form in $\mathrm{Fil}^{k-1}\epsilon_{Y_k}H^{2k-3}_\mathrm{dR}(Y_k/K_\p)$.
\end{itemize}
We also view $\mathscr{L}_{\p,\psi}(f)$ as an element in $\hat{\cO}^\mathrm{ur}\llbracket \mathcal{G}_{p^\infty}\rrbracket$, where $\hat{\cO}^\mathrm{ur}$ denotes the ring of integers of the $p$-adic completion of the maximal unramified extension of $\Q_{f,\p}$.

Let $F$ be a finite extension of $\Q_p$ containing the Fourier coefficients of $f$ and the values of $\hat\psi$, with ring of integers $\cO_F$, and $V_f\simeq F^2$ be the Galois representation attached to $f$. Since $V_f\vert_{G_{\Q_p}}$ is crystalline and $V_f$ ordinary at $p$, the equality \eqref{Gross-Zag} has a $\Lambda$-adic version as proved in \cite[Theorem 7.2]{Magrone} in which the first argument of the pairing is replaced by the value at (the localization at $\p$ of) the big Heegner class $\mathbf{z}_{1,\alpha}$ of Perrin Riou's big logarithm map. On the other hand, the existence of a Perrin-Riou's big exponential map and of the explicit reciprocity law allows us to express the $p$-adic $L$-function $\mathscr{L}_{\p,\psi}(f)$ directly in terms of the class $\mathbf{z}_{1,\alpha}$.

Let $\mathscr{H}_\infty(\mathcal{G}_{p^\infty})=\bigcup_{h\geq 1}\mathscr{H}_{h,F}(\mathcal{G}_{p^\infty})$ where $\mathscr{H}_{h,F}(\mathcal{G}_{p^\infty})=\mathscr{H}_{h,F}(\Gamma_\infty)[\Delta]$ and $\mathscr{H}_{h,F}(\Gamma_\infty)$ is the Perrin-Riou's ring
$$\mathscr{H}_{h,F}(\Gamma_\infty)=\left\{ \sum_{n\geq 0} a_n (\gamma_\infty -1)^n \in F[[\gamma_\infty-1]] \bigg\vert \lim_{n\to \infty} n^{-h}|a_n|_p=0\right\}.$$ 
Denote by $\mathscr{K}_\infty(\mathcal{G}_{p^\infty})$ the total quotient ring of $\mathscr{H}_\infty(\mathcal{G}_{p^\infty})$.

Let $\hat\psi\colon \mathcal{G}_{p^\infty}\rightarrow \cO_F^\times$ be the $p$-adic avatar of a Hecke character $\psi$ as before. By \cite[Eq. 4.34]{Kob-Ota} and \cite[Theorem 7.2]{Magrone}, there exist a unit $u\in \hat{\mathcal{O}}_F^\mathrm{ur}\llbracket \mathcal{G}_{p^\infty}\rrbracket^\times$ and an element 
$$\mathbf{w}\in H^1_\mathrm{Iw}(K[1]_\p,T)\otimes_{\cO_F \llbracket\hat{\mathcal{G}}_{p^\infty}\rrbracket} \mathscr{K}_\infty(\mathcal{G}_{p^\infty}),$$ 
where all the completions in $H^1_\mathrm{Iw}(K[1]_\p,T)=\varprojlim_n H^1(K[p^n]_\p,T)$ are taken with respect to the primes induced by $\iota_p$ and $\hat{\mathcal{G}}_{p^\infty}=\Gal(K[p^\infty]_\p/K_\p)$, such that
\begin{equation}\label{L-funct}
u\langle \mathrm{loc}_\p (\mathbf{z}_{1,\alpha}\otimes \hat\psi^{-1}),\mathbf{w}\rangle = \mathscr{L}_{\p,\psi}(f)\ \in \hat{\mathcal{O}}_F^\mathrm{ur}\llbracket \mathcal{G}_{p^\infty}\rrbracket
\end{equation}
where the pairing is the one defined in \cite[\S 3.6.1]{PR}.
\subsection{Twisted representations}
For a $G_K$-module $B$ over a ring $R$ and a continuous character $\rho\colon G_K\rightarrow R$, denote by $B(\rho)$ the $G_K$-representation over $R$ given by
\[ B\otimes_{R}R_\rho \]
where $R_\rho$ is a free rank-1 $R$-module on which $G_K$ acts via $\rho$. Note that $B(\rho)$ and $B$ are isomorphic as $R$-modules but not as $G_K$-modules. 

Put $\Lambda=\cO_F\llbracket \Gamma_\infty\rrbracket$, and assume that $\hat\psi$ factors through $\Gamma_\infty$. 

Let $\Lambda^\mathrm{ur}$ be the Iwasawa algebra $\hat{\cO}_F^\mathrm{ur}\llbracket \Gamma_\infty\rrbracket$, and denote by $\mathscr{L}_{\p,\psi}^{\mathrm{ac}}$ the image of $\mathscr{L}_{\p,\psi}(f)$ under the natural projection $\hat{\cO}_F^\mathrm{ur}\llbracket \mathcal{G}_{p^\infty}\rrbracket\rightarrow \Lambda^\mathrm{ur}$. More generally, for any $\Lambda$-module $M$ we will use $M^\mathrm{ur}$ to denote $M\otimes_{\cO_F} \hat{\cO}_F^\mathrm{ur}$.

Let $\mathbf{h}$ be an element of $H^1_{\mathcal{F}_\Lambda}(K_\p,\mathbf{T})(\hat\psi^{-1})=H^1_{\mathcal{F}_\Lambda}(K_\p,\mathbf{T})\otimes \hat\psi^{-1}\vert_{G_{K_\p}}$. Then, $\res_{K[1]_\p/K_\p}(\mathbf{h})\in H^1_\mathrm{Iw}(K[1]_\p,\mathrm{Fil}_\p^+(T))(\hat\psi^{-1})$ and by \cite[Eq. 4.33]{Kob-Ota} and \cite[\S 7.1.1]{Magrone}
\[
\langle \res_{K[1]_\p/K_\p}(\mathbf{h}),\mathbf{w}\rangle\in\hat{\cO}_F^\mathrm{ur}\llbracket\mathcal{G}_{p^\infty}\rrbracket. 
\]
Define a homomorphism
\[
\varphi_{\mathbf{w}}\colon H^1_{\mathcal{F}_\Lambda}(K_\p,\mathbf{T})(\hat\psi^{-1})^\mathrm{ur}\longrightarrow \Lambda^\mathrm{ur}
\]
as the composite of $\mathbf{h}\mapsto \langle \res_{K[1]_\p/K_\p}(\mathbf{h}),\mathbf{w}\rangle$ with the natural projection $\hat{\cO}_F^\mathrm{ur}\llbracket\mathcal{G}_{p^\infty}\rrbracket\rightarrow \Lambda^\mathrm{ur}$. It is a $\Lambda$-linear map and

\begin{lemma}\label{varphi}
$\varphi_{\mathbf{w}}(\mathrm{loc}_\p(\kappa_1)\otimes \hat\psi^{-1})=c\cdot \mathscr{L}_{\p,\psi}^\mathrm{ac}$ for some $c\in \Lambda^\times$.
\end{lemma}
\begin{proof}
This follows from the definitions, equation \eqref{L-funct} and Lemma \ref{key-eq}.
\end{proof}

We now consider the twisted representations $T(\hat\psi^{-1})$ and $A(\hat\psi^{-1})$. Then, we write
\[ \mathbf{T}(\hat\psi^{-1}):= \varprojlim \mathrm{Ind}_{K_m/K}(T(\hat\psi^{-1}))\simeq T(\hat\psi^{-1})\otimes \Lambda,\quad \mathbf{A}(\hat\psi^{-1}):=\varinjlim \mathrm{Ind}_{K_m/K}(A(\hat\psi^{-1})).\]

We define a Greenberg-type Selmer structure on $\mathbf{T}(\hat\psi^{-1})$ and $\mathbf{A}(\hat\psi^{-1})$ that we continue to denote by $\mathcal{F}_\Lambda$ abusing notation.
Observe that like $V$, also the twisted Galois representation $V(\hat\psi^{-1})$ satisfies the Panchishkin condition (see \cite[Definition 7.2]{Loe-Zer}). More precisely, since the restriction of $V(\hat\psi^{-1})$ to $G_{K_\p}$ has Hodge-Tate weights $k$ and $1$, the Panchishkin subrepresentation $\mathrm{Fil}_\p^+(V(\hat\psi^{-1}))$ of $V(\hat\psi^{-1})$ at $\p$ is the whole $V(\hat\psi^{-1})$. On the other hand, the restriction of $V(\hat\psi^{-1})$ to $G_{K_{\bar\p}}$ has Hodge-Tate weights $0$ and $1-k$ and hence $\mathrm{Fil}_{\bar\p}^+(V(\hat\psi^{-1}))=0$. Therefore $\mathrm{Fil}_\p^+(T(\hat\psi^{-1}))=\mathrm{Fil}_\p^+(V(\hat\psi^{-1}))\cap T(\hat\psi^{-1})=T(\hat\psi^{-1})$ and $\mathrm{Fil}_\p^+(A(\hat\psi^{-1}))=\mathrm{Fil}_\p^+(V(\hat\psi^{-1}))/\mathrm{Fil}_\p^+(T(\hat\psi^{-1}))=A(\hat\psi^{-1})$.
\begin{remark}
The pairing \eqref{pairing2} induces a pairing
\[ (\ ,\ )_\psi\colon T(\hat\psi^{-1})\times A(\hat\psi^{-1})\longrightarrow \mu_{p^\infty},\]
but $\mathrm{Fil}_v^+(T(\hat\psi^{-1}))$ and $\mathrm{Fil}_v^+(A(\hat\psi^{-1}))$ are not orthogonal to each other with respect to this pairing as it happens in the non-twisted setting.
\end{remark}

As we did in section \ref{Greenberg Lamb-adic}, for $v\mid p$ we set $\mathrm{Fil}^\pm_v(\mathbf{T}(\hat\psi^{-1})):=(\mathrm{Fil}^\pm_v(T(\hat\psi^{-1}))\otimes \Lambda$ and $\mathrm{Fil}^\pm_v(\mathbf{A}(\hat\psi^{-1})):=\varinjlim\mathrm{Ind}_{K_m/K}(\mathrm{Fil}^\pm_v(A(\hat\psi^{-1}))$. 

Let $M$ be $\mathbf{T}(\hat\psi^{-1})$ or $\mathbf{A}(\hat\psi^{-1})$. The Selmer structure $\mathcal{F}_\Lambda$ on $M$ is given by taking the unramified subgroup $H^1_\mathrm{ur}(K_v,M)$ at the primes $v\nmid p$, and $H^1_{\mathcal{F}_\Lambda}(K_v,M)=\mathrm{Im}(H^1(K_v,\mathrm{Fil}^+_v(M)\rightarrow H^1(K_v,M))$ at $v=\p$ and $v=\bar\p$.

We now introduce a more explicite notation: we write $H^1_{\mathcal{F},\mathcal{G}}(K,M)$ for the Selmer group associated with the local conditions $H^1_\mathcal{F}(K_\p,M)$, $H^1_\mathcal{G}(K_{\bar\p},M)$ and $H^1_{\mathrm{ur}}(K_v,M)$ for every prime $v$ of $K$ not dividing $p$. If in particular $H^1_\mathcal{F}(K_\p,M)=H^1(K_\p,M)$ we write $H^1_{\emptyset,\mathcal{G}}(K,M)$ for the associated Selmer group (note that in this case we haven't put any condition on the localization at $\p$), and if $H^1_\mathcal{G}(K_{\bar\p},M)=\{0\}$ we denote the associated Selmer group by $H^1_{\mathcal{F},0}(K,M)$.

By definition, we have 
\[ H^1_{\mathcal{F}_\Lambda}(K,\mathbf{T}(\hat\psi^{-1}))=H^1_{\emptyset,0}(K,\mathbf{T}(\hat\psi^{-1}))=H^1_{\emptyset,0}(K,\mathbf{T})(\hat\psi^{-1}), \]
and the same if we replace $\mathbf{T}(\hat\psi^{-1})$ by $\mathbf{A}(\hat\psi^{-1})$. Moreover as observed in \cite[Remark 2.6]{Kob-Ota}
\[
H^1_{\emptyset,0}(K, \mathbf{A}(\hat\psi^{-1}))=H^1_\mathrm{BK}(K_\infty, A(\hat\psi^{-1})).\]

\begin{theorem}
The $\Lambda$-module $X_\psi=H^1_{\emptyset,0}(K,\mathbf{A}(\psi^{-1}))^\vee$ is torsion, and the following inclusion of ideals of $\Lambda^\mathrm{ur}$ holds:
\[ (\mathscr{L}_{\p,\psi}^\mathrm{ac})^2\subseteq \mathrm{char}(X_\psi)^\mathrm{ur}. \]
\end{theorem}
\begin{proof}
Since $\mathscr{L}_{\p,\psi}^\mathrm{ac}$ is non-zero in $\Lambda^\mathrm{ur}$, $\mathrm{loc}_\p(\kappa_1)\otimes\hat\psi^{-1}$ is not a torsion element. By our main theorem, $H^1_{\mathcal{F}_\Lambda}(K,\mathbf{T})$ is free of rank $1$, and hence $H^1_{\mathcal{F}_\Lambda}(K,\mathbf{T})(\hat\psi^{-1})$ is so. Denote by $\tilde {H}^1_{\mathcal{F}_\Lambda}(K_\p,\mathbf{T})(\hat\psi^{-1})$ the free rank-$1$ $\Lambda$-module $H^1_{\mathcal{F}_\Lambda}(K_\p,\mathbf{T})(\hat\psi^{-1})/H^1_{\mathcal{F}_\Lambda}(K_\p,\mathbf{T})_\mathrm{tors}(\hat\psi^{-1})$, then the localization map
\[ \mathrm{loc}{_\p}\colon H^1_{\mathcal{F}_\Lambda}(K,\mathbf{T})(\hat\psi^{-1}) \longrightarrow \tilde H^1_{\mathcal{F}_\Lambda}(K_\p,\mathbf{T})(\hat\psi^{-1}) \]
is a non-zero map between free modules of rank $1$. In particular, it is injective and it induces an injective map on the quotients
\[ \frac{H^1_{\mathcal{F}_\Lambda}(K,\mathbf{T})(\hat\psi^{-1})}{\Lambda\cdot (\kappa_1\otimes \hat\psi^{-1})}\longmono \frac{\tilde H^1_{\mathcal{F}_\Lambda}(K_\p,\mathbf{T})(\hat\psi^{-1})}{\Lambda\cdot (\mathrm{loc}_\p(\kappa_1)\otimes\hat\psi^{-1})}, \]
whose cokernel is the torsion $\Lambda$-module $\mathrm{coker}(\mathrm{loc}_\p)$.

By Lemma \ref{varphi}, the map induced by $\varphi_\mathbf{w}$
\[ \frac{\tilde H^1_{\mathcal{F}_\Lambda}(K_\p,\mathbf{T})(\hat\psi^{-1})^\mathrm{ur}}{\Lambda^\mathrm{ur}\cdot (\mathrm{loc}_\p(\kappa_1)\otimes\hat\psi^{-1})} \longrightarrow \frac{\Lambda^\mathrm{ur}}{\Lambda^\mathrm{ur}\cdot \mathscr{L}_{\p,\psi}^\mathrm{ac}} \]
is injective. It follows that
\begin{equation}\label{1} 
(\mathscr{L}_{\p,\psi}^\mathrm{ac})\subseteq\mathrm{char}\left(\frac{\Lambda^\mathrm{ur}}{\mathscr{L}_{\p,\psi}^\mathrm{ac}\cdot\Lambda^\mathrm{ur}}\right)\subseteq \mathrm{char}\left(\frac{H^1_{\mathcal{F}_\Lambda}(K,\mathbf{T})(\hat\psi^{-1})^\mathrm{ur}}{\Lambda^\mathrm{ur}\cdot (\kappa_1\otimes\hat\psi^{-1})}\right)\cdot\mathrm{char}(\mathrm{coker}(\mathrm{loc}_\p)^\mathrm{ur}).
\end{equation}
By twisting Theorem \ref{structure}(3), we have
\begin{equation}\label{2}
\mathrm{char}\left(\left(\frac{H^1_{\mathcal{F}_\Lambda}(K,\mathbf{T})(\hat\psi^{-1})}{\Lambda\cdot (\kappa_1\otimes\hat\psi^{-1})}\right)^2\right)  \subseteq \mathrm{char}((H^1_{\mathcal{F}_\Lambda}(K,\mathbf{A})_\mathrm{tors}^\vee)(\hat\psi^{-1}))
\end{equation}

By \cite[Lemma 5.3]{Kob-Ota} and the short exact sequences coming from global duality that appear in its proof, after twisting by $\hat\psi^{-1}$, one has $X_\psi=H^1_{\emptyset,0}(K,\mathbf{A})^\vee(\hat\psi^{-1})$ is torsion and 
\[ \mathrm{char}(X_\psi)=\mathrm{char}(H^1_{\mathcal{F}_\Lambda}(K,\mathbf{A})_{\mathrm{tors}}^\vee(\hat\psi^{-1})\cdot \mathrm{char}(\mathrm{coker}(\mathrm{loc}_\p)^2).\]
Hence, from \eqref{1} and \eqref{2} we get
\[ (\mathscr{L}_{\p,\psi}^\mathrm{ac})^2\subseteq \mathrm{char}(X_\psi)^\mathrm{ur}\]
as claimed.
\end{proof}

\bibliographystyle{amsplain}
\bibliography{Biblio}

\end{document}